\documentclass[11pt]{amsart}
\usepackage[dvips]{graphicx,color}
\usepackage{amssymb,amscd,latexsym,epsfig,xypic,hyperref}
\usepackage[mathscr]{euscript}

\DeclareFontFamily{OT1}{rsfs}{}
\DeclareFontShape{OT1}{rsfs}{n}{it}{<-> rsfs10}{}
\DeclareMathAlphabet{\curly}{OT1}{rsfs}{n}{it}

\makeatletter
\newcommand{\eqnum}{\refstepcounter{equation}\textup{\tagform@{\theequation}}}
\makeatother

\renewcommand\;{\hspace{.6pt}}
\renewcommand\P[1]{{\mathbb P}^{\;#1}}
\newcommand\PP{\mathbb P}
\newcommand\LL{\mathbb L}

\newcommand\C{\mathbb C}
\newcommand\Q{\mathbb Q}

\newcommand\Z{\mathbb Z}

\renewcommand\t{\mathfrak t}

\newcommand\cB{\mathcal B}

\newcommand\cE{\mathcal E}
\newcommand\cF{\mathcal F}

\newcommand\cM{\mathcal M}
\newcommand\cN{\mathcal N}
\newcommand\cO{\mathcal O}
\newcommand\cP{\mathcal P}

\newcommand\udot{^{\bullet}}
\newcommand\E{\mathsf E}
\newcommand\EE{\curly E}

\newcommand\tP{\widetilde P}

\makeatletter
\newcommand{\so}{\ \ext@arrow 0359\Rightarrowfill@{}{\hspace{3mm}}\ }
\makeatother
\newcommand{\rt}[1]{\xrightarrow{\ #1\ }}
\newcommand\To{\longrightarrow}

\newcommand\into{\hookrightarrow}
\newcommand\INTO{\ \ar@{^(->}[r]<-.2ex>}
\newcommand{\Into}{\ensuremath{\lhook\joinrel\relbar\joinrel\rightarrow}}
\newcommand\Mapsto{\ensuremath{\shortmid\joinrel\relbar\joinrel\rightarrow}}

\renewcommand\_{^{}_}
\newcommand\take{\backslash}
\newcommand{\mat}[4]{\left(\begin{array}{cc} \!\!#1 & #2\!\! \\ \!\!#3 &
#4\!\!\end{array}\right)}

\newfont{\bigtimesfont}{cmsy10 scaled \magstep5}
\newcommand{\bigtimes}{\mathop{\lower0.9ex\hbox{\bigtimesfont\symbol2}}}

\newcommand\VW{\mathsf{VW}}
\newcommand\vw{\mathsf{vw}}
\newcommand\uvw{\widetilde{\mathsf{vw}}}
\newcommand\Gr{\operatorname{Gr}}

\newcommand\rk{\operatorname{rank}}
\newcommand\vir{\operatorname{vir}}
\newcommand\vd{\operatorname{vd}}
\newcommand\tr{\operatorname{tr}}

\newcommand\id{\operatorname{id}}
\newcommand\ev{\operatorname{ev}}

\newcommand\Hom{\operatorname{Hom}}
     
\renewcommand\hom{\curly H\!om}
\newcommand\End{\operatorname{End}}
\newcommand\Ext{\operatorname{Ext}}
\newcommand\ext{\curly Ext}
\newcommand\aut{\mathfrak{aut}}
\newcommand\Aut{\operatorname{Aut}}
\newcommand\Pic{\operatorname{Pic}}
\newcommand\Jac{\operatorname{Jac}}

\newcommand\Spec{\operatorname{Spec}\,}
\newcommand\Hilb{\operatorname{Hilb}}

\newcommand\arXiv[1]{\href{http://arxiv.org/abs/#1}{arXiv:#1}}
\newcommand\mathAG[1]{\href{http://arxiv.org/abs/math/#1}{math.AG/#1}}
\newcommand\hepth[1]{\href{http://arxiv.org/abs/hep-th/#1}{hep-th/#1}}

\DeclareRobustCommand{\SkipTocEntry}[3]{}
\makeatletter
\newcommand\@dotsep{4.5}
\def\@tocline#1#2#3#4#5#6#7{\relax
  \ifnum #1>\c@tocdepth 
  \else
    \par \addpenalty\@secpenalty\addvspace{#2}%
    \begingroup \hyphenpenalty\@M
    \@ifempty{#4}{%
      \@tempdima\csname r@tocindent\number#1\endcsname\relax
    }{%
      \@tempdima#4\relax
    }%
    \parindent\z@ \leftskip#3\relax \advance\leftskip\@tempdima\relax
    \rightskip\@pnumwidth plus1em \parfillskip-\@pnumwidth
    #5\leavevmode #6\relax
    \leaders\hbox{$\m@th
      \mkern \@dotsep mu\hbox{.}\mkern \@dotsep mu$}\hfill
    \hbox to\@pnumwidth{\@tocpagenum{#7}}\par
    \nobreak
    \endgroup
  \fi}
\makeatother

\newcommand\beq[1]{\begin{equation}\label{#1}}
\newcommand\eeq{\end{equation}}
\newcommand\beqa{\begin{eqnarray*}}
\newcommand\eeqa{\end{eqnarray*}}

\makeatletter \@addtoreset{equation}{section} \makeatother
\renewcommand{\theequation}{\thesection.\arabic{equation}}
\newtheorem{defn}[equation]{Definition}
\newtheorem{conj}[equation]{Conjecture}

\newtheorem{thm}[equation]{Theorem}
\newtheorem{lem}[equation]{Lemma}

\newtheorem{prop}[equation]{Proposition}
\newtheorem{rmk}[equation]{Remark}

\title[Vafa-Witten invariants II: semistable case]{Vafa-Witten invariants for projective surfaces II: semistable case}
\author[Y. Tanaka and R. P. Thomas]{Yuuji Tanaka and Richard P. Thomas}

\begin{document}
\maketitle \vspace{-5mm}
\begin{center}
\emph{Dedicated to Simon Donaldson, with admiration and thanks}
\end{center}

\begin{abstract} \noindent
We propose a definition of Vafa-Witten invariants counting semistable Higgs pairs on a polarised surface.
We use virtual localisation applied to Mochizuki/Joyce-Song pairs.

For $K_S\le0$ we expect our definition coincides with an alternative definition using weighted Euler characteristics. We prove this for $\deg K_S<0$ here, and it is proved for $S$ a K3 surface in \cite{MT}.

For K3 surfaces we calculate the invariants in terms of modular forms which generalise and prove conjectures of Vafa and Witten.
\end{abstract}
\renewcommand\contentsname{\vspace{-9mm}}
\tableofcontents \vspace{-1cm}


\section{Introduction}

On a polarised surface $(S,\cO_S(1))$ there is a Hitchin-Kobayashi correspondence between solutions of the $U(r)$ Vafa-Witten equations and slope polystable Higgs pairs
$$
(E,\phi), \quad \phi\,\in\,\Hom(E,E\otimes K_S),
$$
with $E$ is a holomorphic bundle 
of rank $r$ on $S$. To partially compactify the moduli space we use Gieseker semistable Higgs sheaves; see Section \ref{Gies} for definitions and \cite[Introduction]{TT1} for a much more detailed account.

Via the spectral construction, Gieseker (semi)stable Higgs pairs $(E,\phi)$ are equivalent to compactly supported Gieseker (semi)stable torsion sheaves $\cE$ on $X=K_S$. That is, given $(E,\phi)$, the sheaf of eigenspaces of $\phi$ --- supported over their respective eigenvalues in $K_S$ --- defines a torsion sheaf
$$
\cE_\phi \quad\mathrm{on}\ \ X=K_S.
$$
Letting $\pi\colon X=K_S\to S$ be the projection, the inverse construction is
$$
(E,\phi)\ =\ \Big(\pi_*\;\cE_\phi\;,\ \pi_*(\eta\cdot\id_{\cE_\phi})\Big),
$$
where $\eta$ is the tautological section of $\pi^*K_S$ on $X=K_S$.

\subsection{Stable case}
Fixing the Chern classes $r,\,c_1,\,c_2\in H^*(S)$ of $E$ on $S$ --- which is equivalent to fixing the topological type of $\cE_\phi$ on $X$ --- 
there is a quasi-projective moduli space $\cN_{r,c_1,c_2}$ of Gieseker semistable Higgs pairs. It is noncompact, but the obvious $\C^*$ action (scaling $\phi$, or equivalently acting on the moduli space of torsion sheaves on $X$ by scaling $K_S$) has projective $\C^*$-fixed locus. When the Chern classes are chosen so that stability and semistability coincide (for instance if the rank and degree of $E$ are coprime) there is a symmetric obstruction theory \cite{TT1} and we can define a $U(r)$ Vafa-Witten invariant by virtual localisation \cite{GP}. It is just a local DT invariant of $X$ counting the stable torsion sheaves $\cE_\phi$.

However, this invariant vanishes unless $H^{0,1}(S)=0=H^{0,2}(S)$. It is much more interesting to consider an analogue of $SU(r)$ Vafa-Witten theory by picking a line bundle $L$ on $S$ and fixing
$$
\det E\ =\ L \quad\mathrm{and}\quad \tr\phi=0
$$
on $S$. On $X$ this amounts to fixing the centre of mass of the support of $\cE_\phi$ on each fibre of $\pi\colon X=K_S\to S$ to be 0, and $\det\pi_*\;\cE_\phi=L$.

In \cite{TT1} it is shown that the resulting moduli space $\cN^\perp_{r,L,c_2}$  also carries a symmetric obstruction theory, so we can define an $SU(r)$ Vafa-Witten invariant by virtual localisation to the compact $\C^*$-fixed locus,
\beq{virtdef}
\VW_{r,L,c_2}\ :=\ \int_{\big[(\cN^\perp_{r,L,c_2})^{\C^*}\big]^{\vir\ }}
\frac1{e(N^{\vir})}\ \in\ \Q.
\eeq
This defines deformation invariant rational numbers whose generating series are expected to give modular forms.

\subsection{Semistable case}
In this paper we describe an extension of this theory in the presence of \emph{strictly semistable Higgs pairs}.
Following ideas of Mochizuki \cite[Section 7.3.1]{Mo} and Joyce-Song \cite{JS}, we rigidify semistable sheaves (which may have nontrivial automorphisms) by taking sections. We define invariants $P^\perp_{r,L,c_2}(n)$ virtually enumerating certain pairs
$$
(\cE_\phi,s)
$$
on $X=K_S$ (or equivalently stable triples $(E,\phi,s)$ on $S$). Here the torsion sheaf $\cE_\phi$ is semistable, has centre of mass zero on the $K_S$ fibres (equivalently $\tr\phi=0$), $\det\pi_*\;\cE_\phi=\det E\cong L$, and
$$
s\ \in\ H^0(X,\cE_\phi(n))\ \cong\ H^0(S,E(n))
$$
does not factor through any subsheaf of $\cE_\phi(n)$ with the same reduced Hilbert polynomial. (Here $n\gg0$ is fixed so that $H^{\ge1}(\cE_\phi(n))=0$ for all semistable sheaves $\cE_\phi$ of class $(r,L,c_2)$.) The moduli space $\cP^\perp_{r,L,c_2}$ admits a symmetric obstruction theory, given by combining the $R\Hom\_\perp$ perfect obstruction theory for $(E,\phi)$ of \cite{TT1} with Joyce-Song's pairs theory; see Section \ref{virtpairs} and particularly \eqref{splits} for more details. We then use virtual $\C^*$ localisation to define invariants 
$$
P^\perp_{r,L,c_2}(n)\ :=\ \int_{\big[(\cP^\perp_{r,L,c_2})^{\C^*}\big]^{\vir}}\ \frac1{e(N^{\vir})}\,.
$$
We conjecture these invariants can be written in terms of universal formulae in $n$ with coefficients given by --- and defining --- Vafa-Witten invariants. Let $\alpha=(r,c_1(L),c_2)$ denote the charge; see \eqref{alph} for a full explanation of the notation.

\begin{conj} \label{pechconj}
If $H^{0,1}(S)=0=H^{0,2}(S)$ there exist $\VW_{\alpha_i}(S)\in\Q$ such that
\beq{desk}
P^\perp_{\alpha}(n)\ =\ \mathop{\sum_{\ell\ge 1,\,(\alpha_i=\delta_i\alpha)_{i=1}^\ell:}}_{\delta_i>0,\ \sum_{i=1}^\ell\delta_i=1}
\frac{(-1)^\ell}{\ell!}\prod_{i=1}^\ell(-1)^{\chi(\alpha_i(n))} \chi(\alpha_i(n))\;\VW_{\alpha_i}(S)
\eeq
for $n\gg0$.
When either of $H^{0,1}(S)$ or $H^{0,2}(S)$ is nonzero we take only the first term in the sum:
\beq{desk2}
P^\perp_{r,L,c_2}(n)\ =\ (-1)^{\chi(\alpha(n))-1}\chi(\alpha(n))\VW_{r,L,c_2}(S).
\eeq
\end{conj}\medskip

The formula \eqref{desk} is copied from Joyce-Song's universal formulae for Joyce-Song pair invariants. (We say something about how to understand it in Section \ref{Ur}.) But Joyce-Song's theory is based on Behrend-weighted Euler characteristics, instead of the virtual cycles we use. Because our moduli spaces are noncompact, there is no \emph{a priori} reason to expect the two theories to behave so similarly. Even the existence of a wall crossing formula like \eqref{desk} in the virtual setting was a surprise to us. The theories \emph{do} diverge when either of $H^{0,1}(S)$ or $H^{0,2}(S)$ is nonzero, with the Joyce-Song formula \eqref{desk} being replaced by what is essentially its logarithm \eqref{desk2}. \medskip

We prove the conjecture in some situations.
\begin{thm} Conjecture \ref{pechconj} holds when stability and semistability coincide; in this case the formulae recover the invariants $\VW_\alpha$ \eqref{virtdef} of \cite{TT1}.
\end{thm}

\begin{thm}
Conjecture \ref{pechconj} holds when $\deg K_S<0$ or when $S$ is a K3 surface \cite{MT}. 
\end{thm}

In Section \ref{K3} we use a wonderful conjecture of Toda \cite{To1}, proved in \cite{MT}, to calculate the invariants $\VW_\alpha$ on K3 surfaces. The generating series are modular, expressed in terms of the Dedekind eta function $\eta$. 

\begin{thm} \label{7}
For $S$ a K3 surface, the generating series of rank $r$ trivial determinant Vafa-Witten invariants equals
\beq{525}
\sum_{c_2}\VW_{r,c_2}q^{c_2}\ =\ \sum_{d|r}\frac{d}{r^2}q^r
\sum_{j=0}^{d-1}\eta\Big(e^{\frac{2\pi ij}d}q^{\frac r{d^2}}\Big)^{-24}.
\eeq
\end{thm}

In particular, when $r$ is prime, so that $d$ takes only the values $1$ and $r$, this recovers a prediction of \cite[End of Section 4.1]{VW},
$$
-\frac1{r^2}q^r\eta(q^r)^{-24}-\frac1rq^r\sum_{j=0}^{r-1}\eta\Big(e^{\frac{2\pi ij}r}q^{1/r}\Big)^{-24}.
$$
Vafa and Witten also asked for the extension to more general $r$, which is precisely what \eqref{525} gives.

\subsection{Behrend localisation} Instead of virtual localisation, one can also consider defining invariants using Behrend localisation. We begin with the case where semistability implies stability.

When a moduli space has a symmetric perfect obstruction theory (as our moduli spaces do) and is compact (in general ours are not) the invariant defined by virtual cycle has another description \cite{Be}. It is the Euler characteristic of the moduli space, weighted by its constructible Behrend function. In the presence of a $\C^*$-action this can be localised to the fixed points, since the other orbits have vanishing (weighted) Euler characteristic.

In our noncompact set up the weighted Euler characteristic defines a \emph{different} invariant which we denote $\vw_\alpha\in\Z$; see Section \ref{modif}. (This is an $SU(r)$ invariant; the corresponding $U(r)$ invariant is denoted $\uvw_\alpha$ and is defined in Section \ref{Ur}.)

One great advantage of $\vw_\alpha$ over the virtual localisation invariant $\VW_\alpha$ \eqref{virtdef} of \cite{TT1} is that the work of Joyce-Song \cite{JS} and Kontsevich-Soibelman \cite{KS} also crucially uses weighted Euler characteristics. Their work therefore applies to $\vw_\alpha$, and allows us --- in Section \ref{modif} --- to extend its definition to the case where there exist strictly semistable Higgs pairs. 

The resulting invariants $\vw_\alpha\in\Q$ need not be deformation invariant, and in \cite{TT1} we showed that in general they give the ``wrong" definition from the point of view of physics. But nonetheless they still give the right answer under certain circumstances.

\begin{thm} \label{mt} When $\deg K_S<0$, or $S$ is a K3 surface \cite{MT}, the two types of Vafa-Witten invariant coincide: $\vw_\alpha=\VW_\alpha$.
\end{thm}

This is one of the results that goes into the proof of Theorem \ref{7}: on a K3 surface we can compute $\vw_\alpha$ instead of $\VW_\alpha$.

\subsection{Modularity} \label{modu}
Vafa and Witten use ``S-duality" to predict that, for fixed rank $r$ and determinant $L$ on any 4-manifold $S$, the \mbox{generating series}
$$
Z_r(S)\ :=\ q^{-s}\sum_{n\in\Z}\,\VW_{r,L,n}(S)\;q^n
$$
should be a weight $w/2=-e(S)/2$ modular form for the shift $s=e(S)/12$.
We confine ourself to $S$ being a projective surface. To calculate $\VW_\alpha$ and confirm such a conjecture, one has to take into account the contribution of two different components of the $\C^*$-fixed locus $\cN^{\;\C^*}_{r,L,c_2}$: \medskip

\begin{enumerate}
\item[\eqnum\label{cpt1}] $\cM_{r,L,c_2}$, the moduli space  of semistable sheaves of fixed determinant $L$ on $S$. These are considered as Higgs pairs by setting $\phi=0$, or, equivalently, as torsion sheaves on $X$ by pushing forward from $S$.
\item[\eqnum\label{cpt2}] We let $\cM_2$ denote the union of \emph{all other} components of $(\cN^\perp_{r,L,c_2})^{\;\C^*}$, i.e. those for which $\phi$ is nilpotent but nonzero. They can be described in terms of flags of sheaves on $S$; when those sheaves have rank one we get the \emph{nested Hilbert schemes of $S$} studied in \cite{GSY1, GSY2, GT} and \cite[Section 8]{TT1}.
\end{enumerate}
When there are strictly semistable Higgs pairs, $\cM_2$ \eqref{cpt2} may not be closed but might touch $\cM_{r,L,c_2}$ \eqref{cpt1}.

The literature has hitherto focussed on only the \emph{first} of these components, and under the restriction that there are no strictly semistables. This implies, by \cite[Section 7.1]{TT1} for instance, that the contribution of \eqref{cpt1} to $\VW_{r,L,c_2}$ is the (integer!) virtual signed Euler characteristic
\beq{signvir}
\int_{[\cM_{r,L,c_2}]^{\vir}}c_{\vd}\big(E\udot\big)\ \in\ \Z,
\eeq
of the instanton moduli space $\cM_{r,L,c_2}$. Here $E\udot\to\LL_{\cM_{r,L,c_2}}$ is the natural obstruction theory, or virtual cotangent bundle, of $\cM_{r,L,c_2}$, and we take its ``virtually top" Chern class $c_{\vd}$, where
$$
\vd\ =\ 2rc_2-(r-1)c_1^2-(r^2-1)\chi(\cO_S)
$$
is its virtual dimension.
This is the Ciocan-Fontanine-Kapranov/Fantechi-G\"ottsche
signed Euler characteristic of $\cM_L$ studied in \cite{JT}.

Similarly $\cM_{r,L,c_2}$'s contribution to $\vw_{r,L,c_2}$
is just its signed topological Euler characteristic
\beq{juste}
(-1)^{\vd(\cM_{r,L,c_2})}e(\cM_{r,L,c_2}),
\eeq
since the Behrend function is $(-1)^{\vd}$ by the dimension reduction result of \cite{BBS, Da} described in \cite[Section 5]{JT}.\footnote{Vafa and Witten would not have the sign $(-1)^{\vd}$ due to different orientation conventions. They identify the tangent and cotangent bundles of $\cM_{r,L,c_2}$ using a Riemannian metric. This is natural from the real point of view, but changes the natural complex orientation that we use by $(-1)^{\dim_\C}$.}

There is a large literature computing these topological Euler characteristics \eqref{juste} and confirming the modularity prediction in examples. Almost all of these example satisfy 
\begin{itemize}
\item $\cM_2=\emptyset$,
\item $\cM_{r,L,c_2}$ contains only \emph{stable} sheaves, and
\item $\cM_{r,L,c_2}$ is smooth.
\end{itemize}
The first and third of these usually follow from a variant of Vafa-Witten's ``\emph{vanishing theorem}" \cite[Section 2.4]{VW}
under some kind of positive curvature condition (such as $K_S\le0$).
The first condition ensures that $\pm e(\cM_{r,L,c_2})$ is the \emph{only} contribution to the Vafa-Witten invariants $\vw$, while the second implies that $\pm e(\cM_{r,L,c_2})$ also equals the virtual invariant \eqref{signvir} in these examples, so $\vw_{r,L,c_2}=\VW_{r,L,c_2}$ and both are integers.

There are four references we know of with computations of Vafa-Witten invariants when no vanishing theorem holds. Firstly \cite{DPS} makes predictions based on modularity. Noncompact surfaces given by line bundles over curves are studied in \cite[Section 3]{AOSV}, while \cite[Section 3]{GGP} studies Vafa-Witten invariants via TQFT for 4-manifolds made by gluing. Most recently \cite{GK} computes the contribution \eqref{signvir} of $\cM_{r,L,c_2}$ to $\VW_{r,L,c_2}$ on general type surfaces, finding modular forms and even refining them by replacing virtual signed Euler characteristics \eqref{signvir} by virtual $\chi_y$ genera. However none of these references calculate on the ``other" component $\cM_2$ \eqref{cpt2}.

Though these references also do not deal with strictly semistable sheaves, Manschot (see \cite{Ma2}, for instance) has long advocated that the correct way to count semistables is using a Joycian formalism (at least in the $K_S<0$ case, where we find $\vw=\VW$). \medskip

Even when all Higgs pairs are stable, the contribution of $\cM_2$ to $\VW_{r,L,c_2}$ can be \emph{rational}, while $\vw_{r,L,c_2}$ is still an integer. In particular, once a vanishing theorem does not hold, $\vw$ and $\VW$ can differ.
In \cite{TT1} we calculated the contribution of $\cM_2$ in examples with only stable Higgs pairs, and in this paper we also calculate with strictly semistables (at which point both invariants $\vw,\,\VW$ become rational numbers.) Such calculations suggest which invariant is ``correct" for physics, as we now explain.

In \cite{TT1} we made some computations of the contributions  of $\cM_2$ to $\VW_\alpha$ on surfaces with $K_S>0$ satisfying some mild conditions (for instance to ensure that semistability implies stability). Generic quintic surfaces, and K3 surfaces blown up in a point are examples. For rank $r=2$, determinant $L=K_S$ and arbitrary $c_2$ there is a natural series of Hilbert schemes $S^{[n]}$ amongst the components of $\cM_2$ \eqref{cpt2} and we managed to sum the generating series of their contributions to the $\VW_\alpha$ into a closed form.

The result was an \emph{algebraic function} of $q$, rather than a modular form. Conversely we found that generating series of Euler characteristic invariants like $\vw_\alpha$ give modular forms up to a factor of $(1-q)^{e(S)}$. Amazingly (to us)\footnote{This issue misled us for some time, as did the (weighted) Euler characteristic calculations of Section \ref{Kthree}, which show that $\vw_\alpha$ gives the right answer on K3 surfaces, even in the presence of semistables. The resolution of this paradox is Theorem \ref{mt}.} however, this is \textbf{not} an indication that $\vw_\alpha$ is preferable to $\VW_\alpha$. On adding the contributions of other components of $\cM_2$ the generating series of $\VW_\alpha$ invariants gives precisely the modular form predicted by Vafa and Witten in low degrees \cite[Section 8]{TT1}. The invariants $\VW_\alpha$ are rational numbers depending on $c_1(K_S)^2$ and $c_2(S)$, whereas the $\vw_\alpha$ integers depend only on the Euler characteristic $c_2(S)$ and give the ``wrong" answers.

\medskip\noindent\textbf{Threefold S-duality.}
Nonetheless we expect there to be a role for \emph{both} definitions $\VW_\alpha$ and $\vw_\alpha$ in the S-duality conjecture for \emph{threefolds} \cite{MSW, GaSY, DeM, dB}. Namely the DT theory of sheaves supported on surfaces in a compact Calabi-Yau 3-fold $X$ should also have modular generating series. These invariants might be expected to be localised to a sum of invariants local to surfaces in $X$. One could then take either form of localisation --- virtual or Behrend --- to recover either type of Vafa-Witten invariant for these surfaces. For each individual surface they might give different answers, but their \emph{sum} over all surfaces in $X$ should give the same modular form when $X$ is compact.

And there has been compelling work showing that the Behrend approach is compatible with threefold S-duality.
In particular Yukinobu Toda has found many modular generating series from weighted Euler characteristics (see for instance his blow-up formula \cite[Theorem 4.3]{To2} and calculations on local $\P2$ \cite{To3} and local K3 \cite{To1}). His work, and that of Manschot et al \cite{Ma1, ABMP}, also shows that wall-crossing transformations (on generating functions of DT invariants counting two dimensional sheaves on Calabi-Yau 3-folds such as $X=K_S$) preserve the modularity predicted by S-duality when one uses weighted Euler characteristics. What is more surprising is that we are predicting that such results should hold (with small modifications like \eqref{desk2}) for invariants defined by virtual localisation too.\medskip

Finally, Emanuel Diaconsecu, Greg Moore, Sergei Gukov and Ed Witten explained to us that Vafa-Witten theory admits a categorification or refinement, given by a topological twist of maximally supersymmetric 5d super Yang-Mills theory. At first sight this would appear to favour $\vw_\alpha$, as it has natural refinements and categorifications using motivic or perverse sheaves of vanishing cycles. But Davesh Maulik has pointed out that $\VW_\alpha$ also admits a natural refinement using $\C^*$-equivariant K-theoretic invariants. And over the locus $\cM_{r,L,c_2}$ of \eqref{cpt1} this recovers the virtual $\chi_y$ genus studied by G\"ottsche-Kool in \cite{GK}. We will explore this refinement in future work.

%

\medskip\noindent\textbf{Acknowledgements.} It is a pleasure to dedicate this paper to Simon Donaldson on the occasion of his 60th birthday. Sir Simon's work has had a profound influence on both of our careers, and his kindness, generosity and humility has had a similar effect on the life of the second author. Fittingly, it was Simon who suggested R.T. have a look at \cite{VW} in 1996; it took a mere 20 years for him to get anywhere with this suggestion. 

Martijn Kool's observations, documented in \cite{TT1}, played a huge role is our recognition that the ``right" invariants should be the $\VW_\alpha$, not $\vw_\alpha$. We also very grateful to Emanuel Diaconescu, Sergei Gukov, Dominic Joyce, Jan Manschot, Davesh Maulik, Greg Moore, Bal\'azs Szendr\H oi, Rahul Pandharipande and Yukinobu Toda for many useful discussions. Finally we thank Ties Laarakker for finding and correcting a mistake in Proposition \ref{stabletrue}.

Y.T. was partially supported by JSPS Grant-in-Aid for Scientific Research numbers JP15H02054 and JP16K05125, and a Simons Collaboration Grant on ``Special holonomy in Geometry, Analysis and Physics". He thanks Seoul National University, NCTS at National Taiwan University, Kyoto University and BICMR at Peking University for support and hospitality during visits in 2015--17 where part of this work was done. \medskip

\noindent\textbf{Notation}
This paper is written in a less formal style than its companion \cite{TT1}. Some parts are conjectural, some parts have proofs which are only sketched, and in parts we are describing a future research programme.
 
We pass backwards and forwards through the spectral construction without comment.
See \cite{TT1} for a detailed review; in particular the equivalence of abelian categories
$$
\mathrm{Higgs}_{K_S}(S)\ \cong\ \mathrm{Coh}_c(X)
$$
between $K_S$-Higgs pairs $(E,\phi)$ on $S$ and compactly supported coherent sheaves on $X$ \cite[Proposition 2.2]{TT1}. This equates Gieseker (semi)stability \eqref{stab1} of the pair $(E,\phi)$ with respect to $\cO_S(1)$ with Gieseker (semi)stability of the sheaf $\cE_\phi$ with respect to $\cO_X(1):=\pi^*\cO_S(1)$.

For rank $r$ and second Chern class $c_2$ we use the notation
$$
\cM_{r,L,c_2}\ \subset\ \cM_{r,c_1,c_2}
$$
for the moduli space of semistable sheaves of  determinant $L$ (respectively first Chern class $c_1=c_1(L)$) on $S$. We reserve $\curly M_S$ for the moduli \emph{stack} of all coherent sheaves on $S$. Similarly
$$
\cN^\perp_{r,L,c_2}\ \subset\ \cN_{r,L,c_2}\ \subset\ \cN_{r,c_1,c_2}
$$
denote moduli spaces of semistable Higgs pairs with Chern classes $r,c_1=c_1(L),c_2$. In the first and second spaces $\det E$ is fixed to be $L$, and in the first space we also impose $\tr\phi=0$. From the second and third of these spaces we will define invariants $\vw,\,\uvw$ by Kai localisation, while from the first we will define $\VW$ by virtual localisation. In the presence of semistables we will work with corresponding spaces of Joyce-Song pairs, denoted by
$$
\cP^\perp_{r,L,c_2}\ \subset\ \cP_{r,L,c_2}\ \subset\ \cP_{r,c_1,c_2},
$$
with corresponding invariants $P^\perp_{r,L,c_2}$ (defined by virtual localisation in Section \ref{virtpairs}) and $P_{r,L,c_2},\,\tP_{r,c_1,c_2}$
(defined by weighted Euler characteristics in Sections \ref{jsP} and \ref{jsP2}) respectively.

\section{Semistable sheaves and Joyce-Song theory}

\subsection{Gieseker (semi)stability}\label{Gies}
We say a Higgs pair $(E,\phi)$ on $(S,\cO_S(1))$ is Gieseker stable if and only if $E$ is pure and --- for every $\phi$-invariant proper subsheaf $F\subset E$ --- there is an inequality of \emph{reduced Hilbert polynomials}
\beq{stab1}
p\_F(n)\,:=\,\frac{\chi(F(n))}{\rk(F)}\ <\ \frac{\chi(E(n))}{\rk(E)}\,=:\,p\_E(n) \quad\mathrm{for\ }n\gg0.
\eeq
Replacing $<$ by $\le$ defines Gieseker semistability.
Gieseker (semi)stability of $(E,\phi)$ is equivalent to Gieseker (semi)stability of the spectral sheaf $\cE_\phi$ with respect to $\cO_X(1)=\pi^*\cO_S(1)$. This is defined by the inequality of reduced Hilbert polynomials
$$
p\_{\cF}(n)\,:=\,\frac{\chi(\cF(n))}{r(\cF)}\ <\ \frac{\chi(\cE_{\phi}(n))}{r(E)}\,=:\,p\_{\cE}(n) \quad\mathrm{for\ }n\gg0,
$$
for all proper subsheaves $\cF\subset\cE$. Here $r(\cE):=\rk(\pi_{*\;}\cE)$ is the leading coefficient of the Hilbert polynomial $\chi\_X(\cE(n))$ divided by$\int_Sc_1(\cO_S(1))^2$.

We fix the Chern classes
$$
\rk(E)=r,\ \ c_1(E)=c_1,\ \ c_2(E)=k
$$
of our Higgs pairs $(E,\phi)$ on $S$. Equivalently, via the spectral construction, we consider compactly supported torsion sheaves on $X$ with rank 0 and
\begin{eqnarray}
c_1 &=& r[S], \nonumber \\
c_2 &=& -\iota_*\Big(c_1+\frac{r(r+1)}2c_1(S)\Big), \label{chern} \\
c_3 &=& \iota_*\Big(c_1^2-2k+(r+1)c_1\cdot c_1(S)+\frac{r(r+1)(r+2)}6c_1(S)^2\Big) \hspace{-1cm} \nonumber
\end{eqnarray}
in $H^*_c(X,\Z)$. Here $\iota\colon S\into X$ is the zero section and $[S]$ its Poincar\'e dual.

We combine these classes into the charge
\beq{alph}
\alpha\ =\ (r,c_1,k)\ \in\ H^{\ev}(S).
\eeq
If we fix $c_1=0$ we often denote this by $\alpha=(r,k)\in H^0(S)\oplus H^4(S)$. The Euler pairing on $X$ of two charges $\alpha,\beta$ is defined to be
$$
\chi(\alpha,\beta)\ :=\ \chi\_X(\cE,\cF)\ =\ \sum(-1)^i\,\mathrm{ext}^i_X(\cE,\cF),
$$
where $\cE,\cF$ are any two torsion sheaves on $X$ whose pushdown to $S$ have charges $\alpha,\beta$ respectively. (Note that we confusingly work on $X$ while expressing charges in terms of data on $S$.) This pairing is skew-symmetric; in particular for any charge $\alpha$,
$$
\chi(\alpha,\alpha)\ \equiv\ 0.
$$
Similarly we have the Hilbert polynomial and reduced Hilbert polynomial of the class $\alpha$,
$$
\chi(\alpha(n))\,:=\,\chi\_X(\cE(n)) \quad\mathrm{and}\quad
p_\alpha(n)\,:=\,\frac{\chi(\alpha(n))}r\,.
$$
We also assume that the polarisation
$\cO_S(1)$ is \emph{generic} so that
\beq{Lgen}
p_\beta(n)\ =\ \mathrm{const}\cdot p_\alpha(n)\ \so\ \beta\ =\ \mathrm{const}\cdot\alpha.
\eeq 
(This assumption restricts the possible sheaves that destabilise $\cE$, and so simplifies the formula \eqref{sum} below. It is purely for simplicity; we can ignore it at the expense of using more complicated formulae from \cite{JS}.)

There is a quasi-projective moduli space parameterising S-equivalence classes of Gieseker semistable sheaves on $X$ with fixed charge $\alpha$. Its $\C^*$-fixed locus is projective, so we would like to define invariants by localising to it. This is no problem when stability and semistability coincide, but in general points of the moduli space represent an entire S-equivalence class of semistable sheaves (rather than a single sheaf), so it is not immediately clear how to count them correctly. 

\subsection{Hall algebra} \label{hall}
Joyce-Song and Kontsevich-Soibelman therefore replace the moduli space by the moduli \emph{stack} of semistable sheaves, and use its Behrend function to define generalised DT invariants which are \emph{rational} numbers in general. 

We describe some of this theory using the formalism of Joyce-Song's noncompact book \cite{JS}. This requires two assumptions that do not always hold when $X=K_S$:
\begin{itemize}
\item $X$ should be ``\emph{compactly embeddable"} \cite[Section 6.7]{JS}, and
\item $H^1(\cO_X)=0$.
\end{itemize}
Both conditions are only used to ensure that moduli of sheaves on $X$ are locally analytically critical loci. The first allows them --- when working with moduli of sheaves --- to pretend that $X$ is compact, while the second makes the line bundle $\cO_X(n)$ spherical. Applying the spherical twist about it (for $n\ll0$) therefore turns moduli of sheaves into moduli of bundles, which can be studied by analytic gauge theoretic methods to prove they are locally analytical critical loci (of the holomorphic Chern-Simons function). This is used to prove identities about Behrend functions.

Team Joyce has since proved that moduli stacks of sheaves on Calabi-Yau 3-folds are always locally algebraic critical loci \cite{BBBJ}, so we can ignore the above conditions. \medskip

Joyce \cite{Jo2} defines a Ringel-Hall algebra. He starts with the $\Q$-vector space on generators given by (isomorphism classes of) morphisms of stacks from algebraic stacks of finite type over $\C$ with affine stabilisers to the stack of objects of Coh$_c(X)$. He then quotients out by the scissor relations for closed substacks. We are interested in the elements
$$
1_{\cN^{ss}_\alpha}\colon\,\cN^{ss}_\alpha\ \Into\ \mathrm{Higgs}_{K_S}(S)\,\cong\,\mathrm{Coh}_c(X),
$$
where $\cN^{ss}_\alpha$ is the \emph{stack} of Gieseker semistable Higgs pairs $(E,\phi)$ of class $\alpha$ on $X$, and $1_{\cN^{ss}_\alpha}$ is its inclusion into the stack of all Higgs pairs on $S$.

To handle the stabilisers of strictly semistable sheaves, Joyce replaces these indicator stack functions by their ``logarithm",
\beq{epsi}
\epsilon(\alpha)\ :=\ \mathop{\sum_{\ell\ge 1,\,(\alpha_i)_{i=1}^\ell:\,\alpha_i\ne0\ \forall i,}}_{p_{\alpha_i}=\,p_{\alpha},\ \sum_{i=1}^\ell\alpha_i=\alpha}\frac{(-1)^\ell}\ell\ 1_{\cN^{ss}_{\alpha_1}}*\cdots*1_{\cN^{ss}_{\alpha_\ell}}\,.
\eeq
In this finite sum $*$ denotes the Hall algebra product. At the level of individual objects, the product of (the indicator functions of) $(E,\phi)$ and $(F,\psi)$ is the stack of all extensions between them,
$$
\frac{\Ext^1(\cF_\psi,\cE_\phi)}{\Aut(\cE_\phi)\times\Aut(\cF_\psi)\times\Hom(\cF_\psi,\cE_\phi)}\,,
$$
with $e\in\Ext^1(\cF_\psi,\cE_\phi)$ mapping to the corresponding extension of $\cF_\psi$ by $\cE_\phi$.
More generally $*$ is defined via the stack $\mathfrak{Ext}$ of all short exact sequences
\beq{exten}
0\To\cE_1\To\cE\To\cE_2\To0
\eeq
in Coh$_c(X)$, with its morphisms $\pi_1,\pi,\pi_2\colon\mathfrak{Ext}\to$ Coh$_c(X)$ taking the extension to $\cE_1,\,\cE,\,\cE_2$ respectively. This defines the universal case, which is the Hall algebra product of Coh$_c(X)$ with itself:
$$
1_{\mathrm{Coh}_c(X)}*1_{\mathrm{Coh}_c(X)}\ =\ \Big(\mathfrak{Ext}\rt\pi\mathrm{Coh}_c(X)\Big).
$$
Other products are defined by fibre product with this: given two stack functions $U,V\to$ Coh$_c(X)$ we define $U*V\to$ Coh$_c(X)$ by the Cartesian square
\beq{Cart}
\xymatrix@C=15pt{
U*V \ar[r]\ar[d]& \mathfrak{Ext} \ar[r]^(.42)\pi\ar[d]_{\pi_1\!}^{\!\times\pi_2}& \mathrm{Coh}_c(X) \\
U\times V \ar[r]& \mathrm{Coh}_c(X)\times\mathrm{Coh}_c(X)\,.\!\!\!}
\eeq
A deep result of Joyce \cite[Theorem 8.7]{Jo3} is that the logarithm \eqref{epsi} lies in the set of \emph{virtually indecomposable stack functions with algebra stabilisers}, $$
\epsilon(\alpha)\ \in\ \bar{\mathrm{SF}}_{\mathrm{al}}^{\mathrm{ind}}(\mathrm{Coh}_c(X),e,\Q).
$$
By \cite[Proposition 3.4]{JS} it can thus be written as a $\Q$-linear combination of morphisms from stacks of the form (scheme)$\,\times B\C^*$, where $B\C^*$ is the quotient stack $(\Spec\C)/\C^*$.
This allows Joyce-Song \cite[Section 5.3]{JS} to take the Kai-weighted Euler characteristic of the stack $\epsilon(\alpha)$ after removing the $B\C^*$ factor (they prove this ``integration map" factors through $\bar{\mathrm{SF}}_{\mathrm{al}}^{\mathrm{ind}}(\mathrm{Coh}_c(X),e,\Q)$). The weighting is by the pullback of the Behrend function $\chi^B$ on Coh$_c(X)$. That is, writing
\beq{epsZ}
\epsilon(\alpha)\ =\ \sum_ic_i\big(f_i\colon Z_i\times B\C^*\To
\mathrm{Coh}_c(X)\big),
\eeq
where the $Z_i$ are \emph{schemes},
\cite[Equation 3.22]{JS} defines generalised DT invariants by 
\beq{JS}
JS_\alpha(X)\ =\ \sum_ic_i\,e\big(Z_i,f_i^*\chi^B\big)\ \in\ \Q.
\eeq
We can localise this invariant. 
The action of $\C^*$ on $X/S$ induces an action on the stack of torsion sheaves by pullback. Similarly pulling back the universal extension over $\mathfrak{Ext}\times X$ we find that if $U,V$ are stacks with $\C^*$ actions and \emph{equivariant} morphisms to Coh$_c(X)$, then the diagram \eqref{Cart} and their Hall algebra product $U*V$ inherit natural $\C^*$ actions. Applied inductively to the $1_{\cN^{ss}_\alpha}$ and their Hall algebra products, we find that $\epsilon(\alpha)$ \eqref{epsi} carries a $\C^*$ action covering that on Coh$_c(X)$.

We claim moreover that in its decomposition \eqref{epsZ}, the pieces $Z_i$ can be taken to be $\C^*$-equivariant. This follows from the proof of the decomposition in \cite{Jo1}, where the key is to use Kresch's stratification of finite type algebraic stacks with affine geometric stabilisers into global quotient stacks \cite[Proposition 3.5.9]{Kr}. This can be done $\C^*$-invariantly, as can the other constructions in \cite[Proposition 5.21]{Jo1}.
 
Since the Behrend function is $\C^*$-invariant, non-fixed $\C^*$ orbits on the $Z_i$ have vanishing weighted Euler characteristic. As a result \eqref{JS} localises to the fixed locus,
\beq{csta}
JS_\alpha(X)\ =\ JS^{\C^*}_\alpha(X)\ :=\ \sum_ic_i\,e\!\left(\!Z_i^{\C^*}\!,f_i^*\chi^B\big|_{Z_i^{\C^*}}\!
\right)\,\in\ \Q.
\eeq
We use these localised invariants of $X$ to define certain $U(r)$ Vafa-Witten invariants of $S$.

\section{$U(r)$ $\uvw$ invariant}
\label{Ur}
\begin{defn} \label{thirdd}
We define a $U(r)$ Vafa-Witten invariant of $S$ by
$$\uvw_{r,c_1,c_2}(S)\ :=\ JS^{\C^*}_{(r,c_1,c_2)}(X)\ \in\ \Q.$$
\end{defn}

When all semistable sheaves are in fact stable this definition reduces to the weighted Euler characteristic of the moduli space or its $\C^*$-fixed locus,
\beq{uvwstab}
\uvw_{r,c_1,k}(S)\ =\ e\Big(\cN_{r,c_1,k},\,\chi^B_{\cN_{r,c_1,k}}\Big)\ =\ e\Big(\cN_{r,c_1,k}^{\C^*},\,\chi^B_{\cN_{r,c_1,k}}\big|_{\cN_{r,c_1,k}^{\C^*}}\Big).
\eeq
But these definitions are only useful when $h^1(\cO_S)=0$ because otherwise the action of $\Jac(S)$ on Coh$_c(X)$ by tensoring forces them to vanish. We define a more useful $SU(r)$ Vafa-Witten invariant $\vw_{r,c_1,c_2}(S)$ in Section \ref{modif}.

While calculating with \eqref{csta} directly is difficult, Joyce and Song prove their invariants may be written more simply in terms of certain \emph{Joyce-Song stable pairs}. We review these next.

\subsection{Joyce-Song pairs}\label{jsP}
Fixing a charge $\alpha$ and $n\gg0$, a Joyce-Song pair $(\cE,s)$ consists of
\begin{itemize}
\item a compactly supported coherent sheaf $\cE$ of charge $\alpha$ on $X$, and
\item a \emph{nonzero} section $s\in H^0(\cE(n))$.
\end{itemize}
We say that the Joyce-Song pair $(\cE,s)$ is \emph{stable} if and only if
\begin{itemize}
\item $\cE$ is Gieseker semistable with respect to $\cO_X(1)$, and
\item if $\cF\subset\cE$ is a proper subsheaf which destabilises $\cE$, then $s$ does \emph{not} factor through $\cF(n)\subset\cE(n)$.
\end{itemize}
For fixed $\alpha$ we may choose $n\gg0$ such that $H^{\ge1}(\cE(n))=0$ for all Joyce-Song stable pairs $(\cE,s)$.
There is no notion of semistability; when $X$ is compact the moduli space $\cP=\cP_{r,c_1,k}(X)$ of stable Joyce-Song pairs is already a projective scheme. It can be shown to be a moduli space of complexes $I\udot:=\{\cO_X(-n)\to\cE\}$ on a Calabi-Yau 3-fold, so $\cP$ carries a symmetric perfect obstruction theory governed by $R\Hom(I\udot,I\udot)\_0$. When $X=K_S$ it may be noncompact, but we can still define \emph{integer} invariants by
$$
\tP_{r,c_1,k}(n)\ :=\ e\Big(\cP_{r,c_1,k}\,,\chi^B_{\cP_{r,c_1,k}}\Big)
$$
and localise them to the $\C^*$-fixed locus:
\beq{welo}
\tP_{r,c_1,k}(n)\ =\ \tP^{\C^*}_{r,c_1,k}(n)\ :=\ e\Big(\cP_{r,c_1,k}^{\C^*},\chi^B_{\cP_{r,c_1,k}}\big|_{\cP_{r,c_1,k}^{\C^*}}\Big).
\eeq

Then for generic polarisation \eqref{Lgen} Joyce-Song's invariants $JS_\alpha(X)=JS_{(r,c_1,k)}(X)\in\Q$ satisfy the following identities \cite[Theorem 5.27]{JS},
\beq{sum}
\tP_{r,c_1,k}(n)\ =\ \mathop{\sum_{\ell\ge 1,\,(\alpha_i=\delta_i\alpha)_{i=1}^\ell:}}_{\delta_i>0,\ \sum_{i=1}^\ell\delta_i=1}
\frac{(-1)^\ell}{\ell!}\prod_{i=1}^\ell(-1)^{\chi(\alpha_i(n))} \chi(\alpha_i(n))\;JS_{\alpha_i}(X).
\eeq
These equations have been simplified by \eqref{Lgen}. If we work with non-generic $\cO_S(1)$, they should be replaced by the full equations of \cite[Theorem 5.27]{JS}. They uniquely determine the $JS_{\alpha}(X)=\uvw_\alpha(S)$, and can be used to define them.

When semistability\,=\,stability for the sheaves $\cE$, the moduli space $\cP_{r,c_1,k}$ is a $\PP^{\;\chi(\alpha(n))-1}$-bundle over the moduli space $\cN_{r,c_1,k}$ of torsion sheaves $\cE$. The Behrend function of $\cP_{r,c_1,k}$ is the pull back of $\cN_{r,c_1,k}$'s, multiplied by the sign $(-1)^{\chi(\alpha(n))-1}$. Therefore taking Euler characteristics and using \eqref{uvwstab} gives
$$
\tP_{r,c_1,k}(n)\ =\ (-1)^{\chi(\alpha(n))-1\,}\chi(\alpha(n))\,\uvw_{r,c_1,k}(S).
$$
This is the first term $\ell=1$ of \eqref{sum}.

More generally the $\ell>1$ terms in \eqref{sum} give rational corrections from semistable sheaves $\cE$. For instance, when $r=2$ and $c_1=0$ the two cases (depending on the parity of $k$) are
$$
\tP_{2,2k+1}(n)\ =\ (-1)^{\chi(\alpha(n))-1\,}\chi(\alpha(n))\,\uvw_{2,2k+1}(S)
$$
when $\alpha=(2,2k+1)$, and
$$
\tP_{2,2k}(n)\ =\ (-1)^{\chi(\alpha(n))-1}\,\chi(\alpha(n))\,\uvw_{2,2k}(S) +\frac12\chi\!\left(\frac\alpha2(n)\right)^2\uvw_{1,k}(S)^2
$$
when $\alpha=(2,2k)$.

\section{$SU(r)$ $\vw$ invariant}\label{modif}
Definition \ref{thirdd} gives $\uvw\equiv0$ when $h^{0,1}(S)>0$ because of the action of $\Jac(S)$ on Coh$_c(X)$ by tensoring. 
So we modify Joyce-Song's theory by fixing the determinant of our sheaves $E=\pi_*\;\cE$. (Even when $h^{0,1}(S)=0$ the resulting $SU(r)$ Vafa-Witten invariant $\vw$ is slightly different from the $U(r)$ invariant $\uvw$ because we remove the deformations $H^0(K_S)$ of the trace of the Higgs field.)

We fix a line bundle $L\in\Pic(S)$ and use the map
$$\xymatrix@=40pt{
\mathrm{Coh}_c(X) \ar[r]^-{\det\circ\,\pi_*}& \Pic(S).}
$$
We denote the fibre over $L$ by
$$
\mathrm{Coh}_c(X)^L\ :=\ (\det\circ\,\pi_*)^{-1}(L).
$$
Then, given any stack function $F:=\big(f\colon U\to$\,Coh$_c(X)\big)$ we can define its fibre over $L\in\Pic(S)$ to be
$$
F^L\ :=\ \big(f\colon U\times\_{\mathrm{Coh}_c(X)}\mathrm{Coh}_c(X)^L\To\mathrm{Coh}_c(X)\big).
$$
This is $1_{\mathrm{Coh}_c(X)^L}\cdot F$, where $\cdot$ is the ordinary (not Hall!) product described in \cite[Definition 2.7]{JS}.

Applied to Joyce's logarithm \eqref{epsZ} we get its fixed determinant analogue
$$
\epsilon(\alpha)^L\,:=\ \sum_ic_i\Big(f_i\colon\big(Z_i\times\_{\mathrm{Coh}_c(X)}\mathrm{Coh}_c(X)^L\big)\big/\C^*\to \mathrm{Coh}_c(X)\Big).
$$
Applying Joyce-Song's integration map to this gives a generalised fixed-determinant DT invariant 
$$
JS^L_\alpha(X)\ :=\ \sum_ic_i\,e\big(Z_i\times\_{\mathrm{Coh}_c(X)}\mathrm{Coh}_c(X)^L,f_i^*\chi^B\big).
$$
As usual, in practice we compute this by localising to $\C^*$-fixed points (and using Joyce-Song pairs in the next Section). As in \eqref{csta} the $\C^*$ action on $X$ covering the identity on $S$ induces $\C^*$ actions on Coh$_c(X)^L$ and $\epsilon(\alpha)^L$. As before we can therefore also take the $Z_i$ to carry equivariant $\C^*$ actions, so that
$$
JS^L_\alpha(X)\ =\ JS^{L,\C^*}_\alpha(X)\,:=\,\sum_ic_i\,e\Big(Z_i^{\C^*}\!\times\_{\mathrm{Coh}_c(X)}\mathrm{Coh}_c(X)^L,f_i^*\chi^B\Big).
$$

\begin{defn} \label{last}
The $SU(r)$ Vafa-Witten invariant is
$$\vw_{r,L,k}(S)\ :=\ (-1)^{h^0(K_S)}JS_{(r,k)}^L(X)\ \in\ \Q.$$
\end{defn}
Here we have inserted the sign to account for the fact that we did not restrict our sheaves $\cE$ to have centre of mass 0 on each fibre of $X\to S$ (equivalently, we did not insist that $\tr\phi=0$) in the construction of $JS^L$. In the $SU(r)$ moduli space this condition should be enforced, and its product with $H^0(K_S)$ (which translates torsion sheaves up the $K_S$ fibres) gives the moduli space we have used. This only affects the Behrend function, and so the weighted Euler characteristic, by the sign $(-1)^{\dim H^0(K_S)}$.

When $h^{0,1}(S)=0$ and $\cO_S(1)$ is generic \eqref{Lgen} this means we modify the pairs theory only by a sign, and \eqref{sum} becomes
\beq{munch}
\tP_{r,c_1,k}(n)\ =\ \mathop{\sum_{\ell\ge 1,\,(\alpha_i=\delta_i\alpha)_{i=1}^\ell:}}_{\delta_i>0,\ \sum_{i=1}^\ell\delta_i=1}
\frac{(-1)^\ell}{\ell!}\prod_{i=1}^\ell(-1)^{\chi(\alpha_i(n))+h^0(K_S)} \chi(\alpha_i(n))\;\vw_{\alpha_i}(S).
\eeq
For $h^{0,1}(S)>0$ we have to modify the pairs theory more significantly.

\subsection{Joyce-Song pairs} \label{jsP2}
We sketch how the Joyce-Song pairs theory gets modified in this fixed-determinant setting. We use the notation of \cite[Chapter 13]{JS}, most of which goes through with only minor modification.
We fix $n\gg0$ and use the same auxiliary categories $\cB_{p_\alpha}$ (whose objects are semistable sheaves $\cE$ with reduced Hilbert polynomial a multiple of $p_\alpha$, plus a vector space $V$ and a linear map $V\to H^0(\cE(n))$), and the same Euler forms $\bar\chi$ thereon. Everything is unchanged up until subsection 13.5, where Joyce-Song apply their integration map (weighted Euler characteristic) to their stack functions $\bar\epsilon_{(\alpha,1)}$ of Equations (13.25) or (13.26). We instead apply their integration map to their fixed determinant analogues.

That is, there is a forgetful map from the stack of objects of $\cB_{p_\alpha}$ to the stack of objects of Coh$_c(X)$, remembering only the sheaf $\cE$. Thus, in their notation, we fix $L\in\Pic(S)$ and define
$$
\bar\epsilon^L_{(\alpha,1)}\ :=\ \bar\epsilon_{(\alpha,1)}
\times\_{\mathrm{Coh}_c(X)}\mathrm{Coh}_c(X)^L.
$$
This is a virtual indecomposible because $\bar\epsilon_{(\alpha,1)}$ is. Applying their integration map $\tilde\Psi^{\cB_{p_\alpha}}$ to it gives the fixed-determinant analogue of their count of Joyce-Song pairs
\beq{plainP}
P_{r,L,k}(n)\,:=\ e\Big(\cP_{r,L,k},\chi^B_{\cP_{r,L,k}}\Big).
\eeq
Here $\cP_{r,L,k}$ is the moduli space of Joyce-Song pairs $(\cE,s)$ with $\det\pi_*\;\cE\cong L$ and charge $\alpha=(r,c_1(L),k)\in H^{\ev}(S)$. As usual the invariant  is calculated in practice by localisation:
$$
P_{r,L,k}(n)\ =\ P^{\C^*}_{r,L,k}(n)\ :=\ 
e\Big((\cP_{r,L,k})^{\C^*}\!,\,\chi^B_{\cP_{r,L,k}}\big|_{(\cP_{r,L,k})^{\C^*}}\Big).
$$
Then applying $\tilde\Psi^{\cB_{p_\alpha}}\big((\ \cdot\ )^L\big)$ to Joyce-Song's equation (13.26) for $\bar\epsilon_{(\alpha,1)}$ gives the fixed-determinant analogue of the formula 
\eqref{sum}. We claim it is the following (much simpler!) formula when $\cO_S(1)$ is generic \eqref{Lgen}.

\begin{prop} \label{pengest} When $h^{0,1}(S)>0$, the invariant $P^{\C^*}_{r,L,k}(n)$ determines the $SU(r)$ Vafa-Witten invariant of Definition \ref{last} by
\beq{be}
P_{r,L,k}(n)\ =\ (-1)^{h^0(K_S)}(-1)^{\chi(\alpha(n))-1\,}\chi(\alpha(n))\,\vw_{r,L,k}(S).
\eeq
\end{prop}

\begin{proof}
We let $\alpha=(r,c_1=c_1(L),k)$ and use \eqref{Lgen} so that the only splittings of $\alpha=\sum_i\alpha_i$ into pieces of the same reduced Hilbert polynomial are of the form $\alpha_i=\delta_i\alpha$ with $\sum_i\delta_i=1$.

This simplifies \cite[Equation 13.26]{JS}.
The first ($\ell=1$) term of the $(\ \cdot\ )^L$ piece gives what we want $P_{r,L,k}(n)$ to be:
\begin{eqnarray} \nonumber
-\tilde\Psi^{\cB_{p_\alpha}}\left(\big[\bar\epsilon_{(0,1)},\,\bar\epsilon_{(\alpha,0)}\big]^L\right)
&=& -\bar\chi\big((0,1),(\alpha,0)\big)\tilde\Psi^{\cB_{p_\alpha}}\big(\bar\epsilon_{(0,1)}\big)\tilde\Psi^{\cB_{p_\alpha}}\big((\bar\epsilon_{(\alpha,0)})^L\big) \\
&=& (-1)^{\chi(\alpha(n))-1\,}\chi(\alpha(n))\,JS_{(r,k)}^L(X) \label{furst}
\end{eqnarray}
by \cite[Proposition 13.13 and Equation 13.30]{JS}.

We will show that the other terms contribute zero by induction on $\ell$. 
The base case is the second ($\ell=2$) term in \cite[Equation 13.26]{JS}, which contributes
\beq{compyou}
\frac12\,\tilde\Psi^{\cB_{p_\alpha}}\!\left(\big[\big[\bar\epsilon_{(0,1)},\,\bar\epsilon_{(\alpha_1,0)}\big],\,\bar\epsilon_{(\alpha_2,0)}\big]^L\right),
\eeq
where $\alpha_i=\delta_i\alpha$ for some $\delta_i$ with $\delta_1+\delta_2=1$. We evaluate this by first pushing down to $\Jac(S)$ via the determinant of the sheaves parameterised by $\bar\epsilon_{(\alpha_1,0)}$, i.e. via
$$
\det\circ\,\pi_*\colon\,\mathrm{Coh}_c(X)_{\alpha_1}\To\Pic_{\delta_1c_1}(S),
$$
before then pushing down to a point. That is, over $M\in\Jac(S)$ we compute
\beq{comp2}
\frac12\,\tilde\Psi^{\cB_{p_\alpha}}\Big[\big[\bar\epsilon_{(0,1)},\,\bar\epsilon_{(\alpha_1,0)}\big]^M,\,\big(\bar\epsilon_{(\alpha_2,0)}\big)^{M^{-1}\otimes L}\Big]
\eeq
--- the contribution of extensions (in both directions) between objects of the first stack (with determinant $M$) and objects of the second with determinant $M^{-1}\otimes L\in\Pic_{\delta_2c_1}(S)$. The result is a constructible function of $M$ whose Euler characteristic we take over $\Pic_{\delta_1c_1}(S)\ni M$ to calculate \eqref{compyou}. 

By \cite[Proposition 3.13 and Equation 13.30]{JS}, \eqref{comp2} is
$$
\frac12\bar\chi\big((\alpha_1,1),(\alpha_2,0)\big)\tilde\Psi^{\cB_{p_\alpha}}\Big(\big[\bar\epsilon_{(0,1)},\,\bar\epsilon_{(\alpha_1,0)}\big]^M\Big)
\tilde\Psi^{\cB_{p_\alpha}}\!\left(\big(\bar\epsilon_{(\alpha_2,0)}\big)^{M^{-1}\otimes L}\right).
$$
We have already seen in \eqref{furst} that the first $\tilde\Psi^{\cB_{p_\alpha}}$ term is independent of $M$. We claim that so is the second.
Therefore the pushdown is a \emph{constant} constructible function on $\Pic_{\delta_1c_1}(S)$. Since $h^{0,1}(S)>0$ this has Euler characteristic zero, and \eqref{compyou} indeed vanishes.\medskip

To prove the claim we use the action (by tensoring) of $\Jac(S)$ on Coh$_c(X)$. This preserves $\chi^B$ since the Behrend function is intrinsic to the stack.

Consider the action of $\Jac(S)$ on the stack $\mathfrak{Ext}$ by tensoring all 3 terms of \eqref{exten} by any $M\in\Jac(S)$. Under the projection $\pi_1\times\pi_2$ it covers the diagonal $\Jac(S)$ action on Coh$_c(X)\times$Coh$_c(X)$, and under the projection $\pi$ it covers the usual $\Jac(S)$ action on Coh$_c(X)$.

Applying this in the diagram \eqref{Cart} we find that given any two stacks $U,V$ with $\Jac(S)$ actions and \emph{equivariant} morphisms $U,V\to$ Coh$_c(X)$, their Hall algebra product $U*V\to$ Coh$_c(X)$ inherits a $\Jac(S)$ action. Applied inductively to the $1_{\cN^{ss}_\alpha}$, we find that the $\epsilon(\alpha)$ and their Hall algebra products all carry a $\Jac(S)$ action covering that on Coh$_c(X)$.
Since the rank of $\alpha$ is $>1$, this $\Jac(S)$ action takes any fibre $(\bar\epsilon_{(\alpha_2,0)})^{L_1}$ of $\bar\epsilon_{(\alpha_2,0)}$ over $L_1\in\Jac(S)$ isomorphically to any other fibre $(\bar\epsilon_{(\alpha_2,0)})^{L_2}$. So their integrals $\Psi^{\cB_{p_\alpha}}$ are the same, because the isomorphism preserves the pullback of $\chi^B$ from Coh$_c(X)$.\medskip

The other terms $\ell\ge3$ vanish for similar reasons. By induction they are of the form (a constant times)
\beq{comp3}
\tilde\Psi^{\cB_{p_\alpha}}\left(\left[F,\,\bar\epsilon_{(\alpha_\ell,0)}\right]^L\right)
\eeq
where $F$ is a stack function taking values in the objects of $\cB_{p_\alpha}$ with charge $((1-\delta_\ell)\alpha,1)$ whose pushdown to $\Pic_{(1-\delta_\ell)c_1}(S)$,
\beq{consta}
M\ \Mapsto\ \tilde\Psi^{\cB_{p_\alpha}}\big(F^M\big)\ \ \mathrm{is\ \emph{constant}}.
\eeq
Now \eqref{comp3} is the Euler characteristic of the constructible function
\beqa
M &\Mapsto& \tilde\Psi^{\cB_{p_\alpha}}\Big[F^M,\big(\bar\epsilon_{(\alpha_\ell,0)}\big)^{M^{-1}\otimes L}\Big] \\
&=& \bar\chi\big(((1-\delta_\ell)\alpha,1),(\alpha_\ell,0)\big)\tilde\Psi^{\cB_{p_\alpha}}\big(F^M\big)
\tilde\Psi^{\cB_{p_\alpha}}\left(\big(\bar\epsilon_{(\alpha_\ell,0)}\big)^{M^{-1}\otimes L}\right),
\eeqa
by \cite[Proposition 3.13 and Equation 13.30]{JS}. By \eqref{consta} this is also constant on $\Pic_{(1-\delta_\ell)c_1}(S)$. Thus its Euler characteristic vanishes.
\end{proof}

\section{K3 surfaces} \label{Kthree}
We first need a foundational result: that $\C^*$-fixed Higgs pairs are in fact $\C^*$-equivariant. For simple (e.g. stable) pairs this is standard --- one can apply \cite[Proposition 4.4]{Ko} to the sheaves $\cE_\phi$, for instance. For pairs with non-scalar automorphisms (e.g. strictly semistable pairs) we have to work a bit harder.

\begin{prop}\label{equi} If $(E,\phi)$ is fixed by the $\C^*$ action scaling $\phi$ then $E$ admits an algebraic $\C^*$ action
$$
\Psi\colon\C^*\To\Aut(E)
$$
such that $\Psi_t\circ\phi\circ\Psi_t^{-1}\ =\ t\phi$ for all $t\in\C^*$.
\end{prop}

\begin{proof}
Since $(E,t\phi)$ must be isomorphic $(E,\phi)$ we get, for each $t\in\C^*$, an automorphism $\psi_t$ of $E$ which conjugates $t\phi$ into $\phi$:
\beq{cstar}
\psi_t\circ\phi\circ\psi_t^{-1}\ =\ t\phi.
\eeq
We will show that the $\psi_t$ may be chosen to define a $\C^*$ action on $E$, i.e. such that $\psi_s\circ\psi_t=\psi_{st}$ for all $s,t\in\C^*$.

\emph{Fix} a $t\in\C^*$ which generates a Zariski dense subset $\{t^n\colon n\in\Z\}$ of $\C^*$. Since $S$ is compact, $\det(\lambda\id-\psi_t)$ is constant on $S$. Thus the eigenvalues $\lambda_j\in\C$ of $\psi_t$ are constant. Let $V_{\lambda_j}=\ker(\psi_t-\lambda_j)^N,\ N\gg1$, be the generalised eigenspaces of $E$, so
\beq{bits}
E\ =\ \bigoplus_jV_{\lambda_j}.
\eeq
Then \eqref{cstar} gives the identities 
\beqa
\psi_t\phi=t\phi\psi_t &\so& (\psi_t-\lambda t)\phi\ =\ t\phi(\psi_t-\lambda) \\
&\so& (\psi_t-\lambda t)^N\phi\,=\,t^N\phi(\psi_t-\lambda)^N.
\eeqa
Applied to $v\in V_{\lambda}$ we get zero, which shows that $\phi v\in V_{t\lambda}$. Therefore $\phi$ acts blockwise on the decomposition \eqref{bits} mapping any $V_{\lambda}$ to $V_{t\lambda}$\;.
Therefore, if we define the block diagonal operator
\beq{tilder}
\widetilde\Psi_t\ =\ \bigoplus_j\lambda_j\id_{V_{\lambda_j}}
\eeq
then this also satisfies $\widetilde\Psi_t\circ\phi\circ\widetilde\Psi_t^{-1}\ =\ t\phi$.

Say that $\lambda_i\sim\lambda_j$ whenever there is some $n\in\Z$ such that $\lambda_i=t^n\lambda_j$. In any equivalence class, choose a representative $\lambda$ and write the other elements as $\lambda_j=t^{\mu_j}\lambda$ for integers $\mu_j$. Then replacing \eqref{tilder} by
\beq{formu}
\Psi_t\ =\ \bigoplus_jt^{\mu_j}\id_{V_{\lambda_j}},
\eeq
this \emph{also} satisfies $\Psi_t\circ\phi\circ\Psi_t^{-1}\ =\ t\phi$. We deduce that $\Psi_{t^n}\circ\phi\circ\Psi_{t^n}^{-1}\ =\ t^n\phi$ for every $n\in\Z$, so by Zariski denseness we conclude that
\[
\Psi_s\circ\phi\circ\Psi_s^{-1}\ =\ s\phi \quad\forall s\in\C^*.
\]
Thus \eqref{formu} defines our required $\C^*$ action on $E$.
\end{proof}

We now show how the $\vw$ theory works on K3 surfaces. We first illustrate the theory with rather explicit calculations in rank 2, before switching to more abstract results in general rank.

\subsection{Rank 2}
We consider semistable rank $r=2$ Higgs sheaves $(E,\phi)$ on a fixed polarised K3 surface $(S,\cO_S(1))$. We fix $\det E=\cO_S$ and $\tr\phi=0$. We use $\t$ to denote the one dimensional $\C^*$ representation of weight 1. 

\begin{lem} \label{ssH}
If $(E,\phi)$ is Gieseker semistable and $\C^*$-fixed, then $E$ is itself Gieseker semistable. Moreover, either $\phi=0$ or $c_2(E)=2k$ is even and (up to an overall twist by some power of $\t$)
\beq{sshp}
E\ =\ I_Z\oplus I_Z\cdot\t^{-1}, \qquad \phi=\mat0010,
\eeq
for some length $k$ subscheme $Z\subset S$.
\end{lem}

\begin{proof}
Let $(E,\phi)$ be $\C^*$-fixed. If $\phi=0$ then $E$ is a semistable $\C^*$-fixed sheaf and we are done. So we now assume that $\phi\ne0$. Then by Proposition \ref{equi} the sheaf $E$ carries a $\C^*$ action acting with weight 1 on $\phi$. Thus $E=\oplus_i E_i$ splits into weight spaces $E_i$ on which $\lambda\in\C^*$ acts as $\lambda^i$.

Since $\phi$ decreases weight it maps the lowest weight torsion subsheaf to zero. This subsheaf is therefore $\phi$-invariant, and so zero by semistability. Therefore each of the $E_i$ are torsion-free, and so in particular have rank$\,>0$. Thus they have rank 1, and there are only two of them:
$$
E\ =\ E_i\oplus E_j,
$$
with $i>j$ without loss of generality. Since the Higgs field has weight 1, it takes weight $k$ to weight $k-1$. It is also nonzero, so we must have $j=i-1$ and the only nonzero component of $\phi$ maps $E_i$ to $E_{i-1}$.

Tensoring $E$ by $\t^{-i}$ (i.e. multiplying the $\C^*$ action on $E$ by $\lambda^{-i}\cdot\id_E$) we may assume without loss of generality that $i=0$ and $j=-1$. Considering $\phi$ as a weight 0 element of $\Hom(E,E)\otimes\t$, we have
\beq{rk1}
E\ =\ E_0\oplus E_{-1}\ \ \mathrm{and}\ \ \phi\ =\ \mat{0}{0}{\Phi}{0}\mathrm{\ \ for\ some\ \ }\Phi\colon E_0\To E_{-1}\cdot\t.
\eeq
Therefore $E_{-1}\subset E$ is $\phi$-invariant, so by the definition of Higgs semistability we get the inequality
$$
\chi(E_0(n))\ \ge\ \chi(E_{-1}(n))  \quad\forall n\gg0.
$$
Since both $E_i$ are torsion free, $\Phi$ is an injection, implying the opposite inequality
$$
\chi(E_0(n))\ \le\ \chi(E_{-1}(n)) \quad\forall n\gg0.
$$
Hence $\Phi$ is actually an isomorphism, and $E=E_0\oplus E_0\cdot\t^{-1}$, which is semistable because $E_0$ is.

Finally, since $E_0$ is rank 1 torsion free with trivial determinant, it is an ideal sheaf $I_Z$, where $Z$ has length $c_2(E)/2$.
\end{proof}

\begin{prop} \label{-1}
Let $(S,\cO_S(1))$ be a K3 surface. At the $\C^*$-fixed points of the stack of rank 2 semistable Higgs pairs $(E,\phi)$ on $S$ with $\det E\cong\cO_S$ and $\tr\phi=0$, we have $\chi^B=-1$.
\end{prop}

\begin{rmk} \label{kaistack} Here our stack has full stabiliser groups, i.e. even simple (e.g. stable) Higgs pairs have stabiliser $\C^*$. If we rigidify by removing multiples of the identity, we change $\chi^B$ to $+1$. If we do not fix $\tr\phi=0$ we change the sign again, since $h^0(K_S)=1$. In particular, the forgetful map from the moduli space of stable Joyce-Song pairs $\cP_\alpha$ in class $\alpha$ to the stack of semistable Higgs pairs (with no condition on $\tr\phi$) is smooth of dimension $\chi(\alpha(n))$, so $(\cP_{\alpha})^{\C^*}$ has Kai function $\equiv(-1)^{\chi(\alpha(n))}$.

We also note that the result of Proposition \ref{-1} is extended to arbitrary rank sheaves on K3 surfaces in \cite{MT}, using completely different methods.
\end{rmk}

\begin{proof} When $\phi=0$ this is a by now well-known result called ``dimension reduction" for $(-1)$-shifted cotangent bundles: the Behrend function is $(-1)^{\vd}$ on its zero section \cite{BBS, Da}, \cite[Section 5]{JT}. Here $\vd$ is the virtual dimension $ext^1(E,E)_0-ext^2(E,E)_0-hom(E,E)=1-\chi(E,E)$ of the moduli stack of sheaves $E$ on $S$ with fixed determinant.

About the other fixed points \eqref{sshp} we need an explicit local model for the moduli stack of objects of $\mathrm{Coh}_c(X)$. On $S$, rather than $X$, the local model is given by \cite[Proposition 3.3]{KaLe}: near $I_Z\oplus I_Z$, the moduli stack $\curly M_S$ of sheaves on $S$ looks like the product of $\Ext^1_S(I_Z,I_Z)$ with the quotient by $GL_2$ of the zero locus of the cup product map
\begin{eqnarray} \label{mc}
\mathfrak{sl}_2\otimes\Ext^1_S(I_Z,I_Z) &\To& \mathfrak{sl}_2 \\
A \quad &\Mapsto& A\cup A \nonumber
\end{eqnarray}
Here $\cup$ denotes the Lie bracket on $\mathfrak{sl}_2$ tensored with the cup product $$\Ext^1_S(I_Z,I_Z)\otimes\Ext^1_S(I_Z,I_Z)\To\Ext^2_S(I_Z,I_Z)\cong\C,$$ and $GL_2$ acts by the adjoint action on $\mathfrak{sl}_2$ and by the identity on $\Ext^1_S(I_Z,I_Z)$.

The Higgs pair \eqref{sshp} is a point of the $(-1)$-shifted cotangent bundle $T^*[-1]\curly M_S$ of $\curly M_S$, with $\phi\in\Hom(E,E)_0\cong\mathfrak{sl}_2$ a point of the fibre over $E=I_Z\oplus I_Z\in\curly M_S$. From \eqref{mc} and the description of $(-1)$-shifted cotangent bundles \cite[Proposition 2.8]{JT} we find a local model for $T^*[-1]\curly M_S$ about $(E,\phi)$. It is the product of $\Ext^1_S(I_Z,I_Z)$ with the quotient by $GL_2$ of the critical locus of the function
\begin{eqnarray} \label{fn}
\mathfrak{sl}_2\otimes\Ext^1_S(I_Z,I_Z)\ \oplus\ \mathfrak{sl}_2 &\To& \quad\C \\
(A,\,\phi)\quad &\Mapsto& 
\tr\big(\phi(A\cup A)\big). \nonumber
\end{eqnarray}
To describe this critical locus, fix a symplectic basis $e_i,f_i$ for $\Ext^1_S(I_Z,I_Z)$. (That is, $\tr(e_i\cup e_j)=0=\tr(f_i\cup f_j)$ and $\tr(e_i\cup f_j)=\delta_{ij}$.) Writing $A\in\mathfrak{sl}_2\otimes\Ext^1_S(I_Z,I_Z)$ as $\sum_iA_{e_i}\otimes e_i+\sum_iA_{f_i}\otimes f_i$ with $A_{e_i},\,A_{f_i}\in\mathfrak{sl}_2$,
the derivative of the function \eqref{fn} down $(a\otimes e_i,0)$ is therefore
$$
2\tr\big(\phi\big[a,A_{f_i}\big]\big)\ =\ -2\tr\big(a\big[\phi,A_{f_i}\big]\big).
$$
The vanishing of this for all $a\in\mathfrak{sl}_2$ is equivalent to the vanishing of $[\phi,A_{f_i}]$. Replacing $e_i$ by $f_i$ we similarly get the vanishing of $[\phi,A_{f_i}]$ for all $i$.
We conclude that at a point with $\phi\ne0$, each $A_{e_i},\,A_{f_i}$ is proportional to $\phi$, i.e.
\beq{zeroa}
A\in\langle\phi\rangle\otimes\Ext^1_S(I_Z,I_Z).
\eeq
In turn this forces the derivative
$$
\tr\big(\psi(A\cup A)\big)
$$
of \eqref{fn} down $(0,\psi)$ to vanish, so \eqref{zeroa} is precisely the critical locus. (With a bit more care, writing out the equations via a basis for $\mathfrak{sl}_2$, one can see that the scheme structure of the critical locus is the reduced one on the locus \eqref{zeroa}.)

In particular we see that for $\phi\ne0$ the critical locus is smooth --- it is an $\Ext^1_S(I_Z,I_Z)$-bundle over $\mathfrak{sl}_2\take\{0\}$. Multiplying by $\Ext^1_S(I_Z,I_Z)$ we get a smooth odd dimensional space whose Behrend function is therefore $-1$. Dividing by $GL_2$, which is even-dimensional, does not change this.
\end{proof}

\noindent ${\mathbf{c_2}}$ {\bf odd.} When $c_2(E)$ is \emph{odd}, by Lemma \ref{ssH} the Higgs field vanishes,  the sheaf $E$ is \emph{stable}, and the moduli space $\cN^\perp_{2,c_2}$ is just the moduli space of instantons on $S$ (pushed forward to $X$). This was observed in \cite{VW} as a case where their vanishing theorem holds. In particular the moduli space is smooth, hyperk\"ahler, has $\chi^B\equiv1$, and is deformation equivalent to
\beq{hil}
\Hilb^{2c_2(E)-3}(S).
\eeq
So by G\"ottsche's formula
$$
\sum_n q^n e(\Hilb^nS)\ =\ \left(\,\prod_{k=1}^\infty\frac1{(1-q^k)}\right)^{\!e(S)}=\ q\,\eta(q)^{-24}
$$
we can evaluate the contribution of odd $c_2$ to the generating function. The result is
\begin{multline} \label{oddd}
\sum_{c_2\ \mathrm{odd}} q^{c_2\,} e(\Hilb^{2c_2(E)-3}S)
\ = \\ \frac14q^2\Big(\eta(q^{1/2})^{-24}+\eta(-q^{1/2})^{-24}-\eta((-q)^{1/2})^{-24}-\eta(-(-q)^{1/2})^{-24}\Big).
\end{multline}

\noindent ${\mathbf{c_2}}$ {\bf even.} 
For $c_2(E)$ even, however, no such vanishing result holds and we have to deal with strictly semistable Higgs pairs. We take a Joycian approach, and compare the result to the predictions of Vafa and Witten. By enforcing modularity of the final result, they conjectured that the generating function of invariants should be
\beq{pred}
\frac14q^2\eta(q^2)^{-24}+\frac12q^2\Big(\eta(q^{1/2})^{-24}+\eta(-q^{1/2})^{-24}\Big).
\eeq
Here we have adjusted for the $\pm1/|Z(G)|=\pm1/2$ difference in our invariants, and omitted Vafa-Witten's shift $q^{-2}$. Putting these back in gives the modular form of \cite[Equation 4.17]{VW}. Taking only odd powers of $q$ in \eqref{pred} recovers \eqref{oddd}.

In particular the prediction \eqref{pred} starts with
$$
\vw_{2,0}(S)=\frac14,\quad \vw_{2,1}(S)=0,\quad \vw_{2,2}(S)=\frac{24}4+24=30,
$$
which we shall now check explicitly as an illustration of the theory.

For the first we use
$$
\vw_{1,0}(S)\ =\ 1,
$$
counting the sheaf $\cO_S$ on $X$ (rigid in the space of sheaves with fixed centre of mass 0 on the fibres of $K_S$), and
\begin{eqnarray} \nonumber
P_{2,0}(n) &=& \frac{\chi(\cO_S(n))\big(\chi(\cO_S(n))-1\big)}2\ +\ \big(2\chi(\cO_S(n))-\chi(\cO_S(n))\big) \\ \label{fa}
&=& \frac14\chi(\alpha(n))+\frac12\chi\!\left(\frac\alpha2(n)\right)^2,
\end{eqnarray}
where $\alpha$ is the class $(2,0)$ of $\cO_S\oplus\cO_S$.
The first term is the Euler characteristic of the moduli space $\Gr(2,H^0(\cO_S(n)))$ of stable Joyce-Song pairs with underlying sheaf $\cE=\cO_S\oplus\cO_S$. For the second the $\C^*$-fixed sheaf is $\cE=\cO_{2S}:=\cO_X/I_{S\subset X}^2$ and the corresponding moduli space of stable Joyce-Song pairs is $\PP\big(H^0(\cO_{2S}(n)\take H^0(\cO_S(n))\big)\big/\C$, with Euler characteristic $2\chi(\cO_S(n))-
\chi(\cO_S(n))$. There is no additional sign, due to the identity $\chi^B_{\cP_{2,0}(X)}\big|_{\cP_{2,0}(X)^{\C^*}}\equiv1$ of Remark \ref{kaistack}.

For a class $\alpha$ with divisibility 2 on a K3 surface $S$, \eqref{munch} reads
\beq{peng}
P_{\alpha}(n)\ =\ \chi(\alpha(n))\vw_\alpha(S)+\frac12\chi\!\left(\frac\alpha2(n)\right)^{\!2}\vw_{\alpha/2}(S)^2
\eeq
Comparing to \eqref{fa} gives $\vw_{2,0}(S)=\frac14$, as required.

The second prediction $\vw_{2,1}(X)=0$ already follows from our analysis \eqref{hil} of the odd $c_2$ case, of course.

So we are left with the third, $\vw_{2,2}(X)=30$.

\begin{lem} Any $\C^*$-invariant semistable sheaf $\cE$ on $X$ of class $(2,2)$ and $\det\pi_*\;\cE\cong\cO_S$ is a strictly semistable extension of the form
$$
0\To\iota_*I_x\To\cE\To\iota_*I_y\To0,
$$
for points $x,y\in S$.
\end{lem}

\begin{proof}
We use Lemma \ref{ssH}. If $\phi\ne0$ then by \eqref{sshp} we see that $\cE$ is $\pi^*I_x\otimes\cO_{2S}$ for some $x\in S$. This is an extension of the form required, with $y=x$.

This leaves $\phi=0$, so that $\cE=\iota_*E$ is determined by the sheaf $E=\pi_{*\,}\cE$. By the  semistability of Lemma \ref{ssH} we have $h^0(E)=0$. But $\chi(E)=2$, so
$$
h^2(E)\ =\ \mathrm{hom}(E,\cO_S)\ \ge\ 2.
$$
So we may pick a nonzero map $\phi\colon E\to\cO_S$. Its image is an ideal sheaf $I\subset\cO_S$ which by the semistability of $E$ can only have cokernel of dimension zero and length $0$ or $1$. It is therefore either $\cO_S$ or $I_y$ for some point $y\in S$.

The kernel of $\phi$ is a rank 1 torsion free sheaf of trivial determinant and so is also an ideal sheaf $I_Z$. Since $c_2(E)=2$ we find $Z$ has length $2$ or $1$ in the two cases above. If the latter it takes the form $I_x$ and we are done. If the former we get an exact sequence
\beq{pmi}
0\To I_Z\To E\To\cO_S\To0
\eeq
with $Z$ of length $2$. Pick a point $y\in Z$ such that $\Hom(I_Z,I_y)=\C$. Since $H^1(I_y)=0$, the long exact sequence of $\Hom(\ \cdot\ ,I_y)$ applied to \eqref{pmi} shows that $\Hom(E,I_y)=\C$, and we can proceed as before.
\end{proof}

So we can now classify $\C^*$-fixed stable Joyce-Song pairs with underlying semistable sheaf $\cE$ in class $\alpha=(2,2)$.
\begin{itemize}
\item $\cE=\iota_*(I_x\oplus I_y)$ with $x\ne y\in S$. The pairs moduli space is a $\PP(H^0(I_x(n)))\times \PP(H^0(I_y(n)))$-bundle over $(S\times S\take\Delta_S)\big/\Z/2$ with Euler characteristic
$$
\chi\!\left(\frac\alpha2(n)\right)^{\!2\,}\frac{e(S)^2-e(S)}2\,.
$$
\item $\cE=\iota_*(I_x\oplus I_x)$ with $x\in S$. The moduli space of Joyce-Song pairs is a $\Gr(2,H^0(I_x(n)))$-bundle over $S$ with Euler characteristic
$$
\frac12\,\chi\!\left(\frac\alpha2(n)\right)\left(\chi\!\left(\frac\alpha2(n)\right)-1\right)e(S).
$$
\item $\cE$ is a \emph{nontrivial extension} between $I_x$ and itself, $x\in S$, classified by a point of
$$
\PP(\Ext^1(I_x,I_x))^{\C^*}\ \cong\ \PP(T_xX)^{\C^*}\ =\ \PP(T_xS)\sqcup S,
$$
where the very last term corresponds to the vertical $K_S$ direction in $T_xX$. Then the pairs space is a $\PP(H^0(\cE(n))\take H^0(I_x(n)))\big/\C$-bundle over $\PP(TS)\sqcup S$ with Euler characteristic
$$
\left(\chi(\alpha(n))-\chi\!\left(\frac\alpha2(n)\right)\right)3e(S).
$$
\end{itemize}
Adding it all up and using Remark \ref{kaistack} gives
$$
P_{2,2}(n)\ =\ \frac54\chi(\alpha(n))e(S)+\frac12\chi\!\left(\frac\alpha2(n)\right)^{\!2}e(S)^2.
$$
Using
$$
\vw_{\alpha/2}(S)\ =\ \vw_{1,1}(S)\ =\ e(S)\ =\ 24,
$$
and comparing to \eqref{peng}
gives $\vw_{2,2}(S)=\frac54e(S)=30$, as required by modularity. \bigskip

\noindent\textbf{All} $\mathbf{c_2.}$
For the general case we use the following conjecture of Toda \cite{To1}, now proved in \cite{MT}:
\beq{toda}
JS_{\alpha}(X)\ =\ -\sum_{k\ge1,\ k|\alpha}\frac 1{k^2}
e\big(\!\Hilb^{1-\frac12\chi\_S\!\big(\!\frac{\alpha}k,\frac\alpha k\!\big)}\!S\big).
\eeq
Here $\alpha$ is any class in $H^*(S,\Z)$ and $X=S\times\C$ as usual, while $\chi\_S$ is the Mukai pairing \emph{on $S$} instead of $X$ (this is minus the pairing Toda uses). We have added a sign to Toda's formula because he uses bare Euler characteristics. He hints at the natural conjecture that $\chi^B$ should be $\pm1$ so that \eqref{toda} gives the correct virtual answer; this was proved in rank 2 in Proposition \ref{-1}, and in general in \cite{MT}.

In particular for $\alpha=(2,2k)$ we get
\beq{2k}
\vw_{2,2k}(S)\ =\ e(\Hilb^{4k-3}S)+\frac14e(\Hilb^kS),
\eeq
while we already know from \eqref{hil} that
\beq{2k1}
\vw_{2,2k+1}(S)\ =\ e(\Hilb^{4k-1}S).
\eeq
So the generating series is
$$
\frac12q^2\big(\eta(q^{1/2})^{-24}+\eta(-q^{1/2})^{-24}\big)+\frac14\eta(q^2)^{-24},
$$
the last term coming from the last term in \eqref{2k}, and the first term coming from the sum of the two remaining terms in (\ref{2k}, \ref{2k1}). But this is precisely the Vafa-Witten prediction \eqref{pred}.\bigskip

\subsection{All rank and all $c_2$} \label{Kvw3}
Given the proof of Toda's conjecture \eqref{toda} in \cite{MT} the analysis in higher rank $r$ is no harder.

\begin{thm} \label{Kvw2}
The generating series of rank $r$, trivial determinant, weighted Euler characteristic Vafa-Witten invariants equals
\beq{generalr}
\sum_{c_2}\vw_{r,c_2}q^{c_2}\ =\ \sum_{d|r}\frac{d}{r^2}q^r
\sum_{j=0}^{d-1}\eta\Big(e^{\frac{2\pi ij}d}q^{\frac r{d^2}}\Big)^{-24}.
\eeq
\end{thm}

\begin{proof}
By \eqref{toda} we find
$$
\sum_n\vw_{r,n}(S)q^n\ =\ \sum_{d|r}\frac1{d^2}\sum_{m\in\Z}
e\Big(\Hilb^{\frac rd\big(m-\frac rd\big)+1}S\Big)q^{md},
$$
where on the right we have summed over those $n=md$ divisible by $d$. Shifting $m$ by the integer $r/d$ and then swapping the roles of $d$ and $r/d$ shows the generating series is
\beq{temp}
\sum_{d|r}\frac{d^2}{r^2}\sum_{m\in\Z}
e\Big(\!\Hilb^{dm+1}S\Big)q^{\frac{mr}d+r}.
\eeq
To sum this we rewrite G\"ottsche's formula as
$$
\sum_n e(\Hilb^{n+1}S)\;q^n\ =\ \eta(q)^{-24}
$$
and take only powers of $q$ divisible by $d$ on both sides to give
$$
\sum_m e(\Hilb^{md+1}S)q^{md}\ =\ \frac1d\sum_{j=0}^{d-1}\eta\Big(e^{\frac{2\pi ij}d}q\Big)^{-24}.
$$
Substituting in \eqref{temp} we find the generating series is
\[
\sum_{d|r}\frac{d}{r^2}q^r
\sum_{j=0}^{d-1}\eta\Big(e^{\frac{2\pi ij}d}q^{\frac r{d^2}}\Big)^{-24}. \qedhere
\]
\end{proof}

When $r$ is prime, so that $d$ takes only the values $1$ and $r$, this becomes
$$
-\frac1{r^2}q^r\eta(q^r)^{-24}-\frac1rq^r\sum_{j=0}^{r-1}\eta\Big(e^{\frac{2\pi ij}r}q^{1/r}\Big)^{-24}.
$$
In \cite[End of Section 4.1]{VW} Vafa and Witten made precisely this prediction, and asked for the extension to more general $r$, which is what \eqref{generalr} gives. Martijn Kool pointed out to us that \eqref{generalr} proves a physics conjecture from \cite[Equation 3.7]{MNVW}.

\section{Extension of $SU(r)\ \VW$ invariant to the semistable case}
\label{virtpairs}
Finally we describe the virtual localisation version $\VW_\alpha$ of the $SU(r)$ Vafa-Witten invariant $\vw_\alpha$ described in Section \ref{modif}.

Fix $n\gg0,\,r>0,\,c_1,\,c_2$ and $L\in\Pic_{c_1}(S)$, and recall the notion of a Joyce-Song pair $(\cE,s)$ from Section \ref{jsP}. So for us, $\cE$ is a pure dimension 2 Gieseker semistable sheaf on $X=K_S$ whose pushdown $E=\pi_{*\;}\cE$ has rank $r$ and Chern classes $c_1,\,c_2$. (Equivalently it is a semistable Higgs pair $(E,\phi)$ on $S$.) Then $s$ is a section of $\cE(n)$ which does not factor through any destabilising subsheaf.
Let
$$
\cP_{r,L,c_2}^\perp\ \subset\ \cP_{r,c_1,c_2}
$$
denote the moduli space of Joyce-Song pairs with $\det E\cong L$ and $\tr\phi=0$.

In \cite[Chapter 12]{JS} Joyce-Song construct a symmetric perfect obstruction theory on $\cP_{r,c_1,c_2}$ which has tangent-obstruction complex
\beq{JSdef}
R\Hom\_X(I\udot,I\udot)\_0[1]
\eeq
at the point $I\udot:=\{\cO_X(-n)\rt{s}\cE\}$. This can be modified to give a symmetric perfect obstruction theory on $\cP^\perp_{r,L,c_2}\subset\cP_{r,c_1,c_2}$ by following \cite[Section 5]{TT1}.\footnote{In \cite[Section 5]{TT1} we modified the obstruction theory for $\cN$ by removing $H^1(\cO_S)\oplus H^2(\cO_S)[-1]\oplus H^0(K_S)\oplus H^1(K_S)[-1]$ from $\tau^{[0,1]}\big(R\Hom_X(\cE,\cE)\_0[1])$ to get an obstruction theory based on $R\Hom_X(\cE,\cE)\_\perp[1]$ for $\cN^\perp_L$.} That is, removing $H^1(\cO_S)\oplus H^2(\cO_S)[-1]\oplus H^0(K_S)\oplus H^1(K_S)[-1]$ from \eqref{JSdef}
gives (the shift by $[1]$ of) the first term of the decomposition
\begin{align} \nonumber
R\Hom\_X(I\udot,I\udot)\ \cong\ &R\Hom\_X(I\udot,I\udot)\_\perp \\ &\qquad\oplus H^*(\cO_X)\oplus H^{\ge1}(\cO_S)\oplus H^{\le1}(K_S)[-1]. \label{splits}
\end{align}
On the second line we have removed the deformation-obstruction theory of $\det(I\udot)\cong\cO_X$, of $\det\pi_{*\;}\cE\cong\cO_S$ and of $\tr\phi\in\Gamma(K_S)$ respectively.
Doing this in a family, shifting by $[1]$ and dualising gives a perfect symmetric obstruction theory on $\cP^\perp_{r,L,c_2}$. 
Full details will appear in a future paper.

\begin{rmk} \normalfont
We sketch a quicker route for those familiar with derived algebraic geometry. Consider the derived stack \cite{TVa} $\curly M_S$ of all torsion-free sheaves on $S$, with determinant map $\curly M_S\to\!\curly J\!ac(S)$ \cite{STV} to the derived stack of line bundles. $(\curly J\!ac(S)$ is a derived stack with underlying scheme the usual Jacobian, stabiliser group $\C^*$ at every point, and obstruction bundle $H^2(\cO_S).)$ The fibre $\curly M^L_S$ over $L\in\!\!\curly J\!ac(S)$ inherits a derived stack structure. Over this we form the stack $\curly P^L_S$ of pairs $(E,s)$, where $E\in\curly M^L_S$ and $s\in H^0(E(n))$. The forgetful map $\curly P^L_S\to\curly M_S^L$ is smooth with fibre $H^0(E(n))$ over the open locus of sheaves with $H^{\ge1}(E(n))=0$, so inherits a natural derived stack structure from $\curly M_S^L$. This induces a derived stack structure on (the corresponding open locus of) its $(-1)$-shifted cotangent bundle $T^*[-1]\curly P^L_S$. This is a moduli stack of triples $(E,s,\phi)$ where $s\in H^0(E(n))$ and $\phi\in\Ext^2(E,E)_0^*\cong\Hom(E,E\otimes K_S)\_0$. Finally, $\cP^\perp_X\subset T^*[-1]\curly P^L_S$ is the open substack of \emph{stable} triples. These have trivial stabiliser groups, vanishing $H^1(E(n))$ and so an induced derived structure which is $(-1)$-symplectic and quasi-smooth. In particular it inherits the symmetric perfect obstruction theory that we are after.
\end{rmk}

So we can now apply virtual localisation to the fixed locus of the $\C^*$ action scaling the $K_S$ fibres of $X\to S$ (equivalently scaling the Higgs fields $\phi$) to define invariants
\beq{bowl}
P^\perp_{r,L,k}(n)\ :=\ \int_{\big[(\cP^\perp_{r,L,c_2})^{\C^*}\big]^{\vir\ }}\frac1{e(N^{\vir})}\,.
\eeq
We speculate that these satisfy similar\footnote{Similar, but nonetheless genuinely different when $h^{0,2}(S)>0$.} identities to the Joyce-Song invariants $P_{r,c_1,k}(n),\ P_{r,L,k}(n)$ defined by Kai localisation in Sections \ref{jsP} and \ref{jsP2}, with $\vw$ replaced by $\VW$. We use the notation \eqref{alph} from Section \ref{hall}, and a generic polarisation $\cO_S(1)$ as in \eqref{Lgen}.

\begin{conj} \label{conj}
If $H^{0,1}(S)=0=H^{0,2}(S)$ there exist rational numbers $\VW_{\alpha_i}(S)$ such such that
\beq{long}
P^\perp_{\alpha}(n)\ =\ \mathop{\sum_{\ell\ge 1,\,(\alpha_i=\delta_i\alpha)_{i=1}^\ell:}}_{\delta_i>0,\ \sum_{i=1}^\ell\delta_i=1}
\frac{(-1)^\ell}{\ell!}\prod_{i=1}^\ell(-1)^{\chi(\alpha_i(n))} \chi(\alpha_i(n))\;\VW_{\alpha_i}(S)
\eeq
for $n\gg0$.
When either of $H^{0,1}(S)$ or $H^{0,2}(S)$ is nonzero we take only the first term in the sum:
\beq{shorter}
P^\perp_{r,L,c_2}(n)\ =\ (-1)^{\chi(\alpha(n))-1}\chi(\alpha(n))\VW_{r,L,c_2}(S).
\eeq
\end{conj}

The formula \eqref{shorter} is of course reminiscent of the formula \eqref{be} for invariants defined by weighted Euler characteristic when $h^{0,1}(S)>0$, but is different for $h^{0,1}(S)=0<h^{2,0}(S)$. The motivation for dropping the other terms when $h^{0,1}(S)>0$ or $h^{0,2}(S)>0$ is that we think of them as enumerating the contributions of nontrivial direct sums of sheaves. When $h^{0,1}(S)>0$ these come in families with a nontrivial $\Jac(S)$ action, for instance with $M\in\Jac(S)$ acting on $\cE_1\oplus\cE_2$ by taking it to the sheaf $\cE_1\otimes M^{-r_2}\oplus\cE_2\otimes M^{r_1}$ with the same determinant. (Here $r_i=\rk(\cE_i)$.) This defines a nowhere zero vector field, and so cosection of the obstruction sheaf, over these semistable loci.  Similarly when $h^{0,2}(S)>0$ these loci inherit extra trivial pieces in their trace-free obstruction spaces. So in both cases we expect their virtual contribution to be zero.

Since \eqref{sum} is a wall crossing formula for invariants defined by weighted Euler characteristic it is natural to expect our conjecture to be proved by an extension of the different wall crossing formula of Kiem-Li \cite{KL2}. Their work uses virtual localisation instead of the Behrend function, so should fit naturally with $\VW_\alpha$. This should also be used to clarify the invariance (expected in physics) of $\VW_\alpha$ under changes in the polarisation $\cO_S(1)$. We intend to return to this in future work.

We start by proving these conjectures --- and showing they recover the invariants $\VW_\alpha\in\Z$ of \eqref{virtdef} --- when stability and semistability coincide.

\begin{prop}\label{stabletrue}
If all semistable sheaves in $\cN^\perp_{r,L,c_2}$ are stable then Conjecture \ref{conj} is true with $\VW_{r,L,c_2}\in\Q$ defined by \eqref{virtdef}.
\end{prop}

\begin{proof}
We sketch the proof using induction on the rank $r$ of $\alpha=(r,c_1,c_2)$. We first claim that if there are no strictly semistables in class $\alpha$ then only the first term contributes to the sum \eqref{long}. Indeed, if there was a nonzero contribution indexed by $\alpha_1,\ldots,\alpha_\ell$ with $\ell>1$ then the nonvanishing of the numbers $\VW_{\alpha_i}(S)$ (which equal the numbers \eqref{virtdef} by the induction hypothesis) would imply that the moduli spaces $\cN^\perp_{\alpha_i}$ are nonempty. Picking an element $\cE^i$ of each defines a strictly semistable $\cE:=\cE^1\oplus\cdots\oplus\cE^\ell$ of $\cN_\alpha^\perp$, a contradiction.

Let $\pi\colon X\times\cN_\alpha^\perp\to\cN_\alpha^\perp$ denote the projection, and let $\EE$ denote the (possibly twisted) universal sheaf. Since there are no strictly semistables, and all stable sheaves are simple,
the moduli space of pairs
\beq{txt}
\cP^\perp_{\alpha}\ =\ \PP(\pi_*\;\EE_n)\rt{p}\cN^\perp_\alpha
\eeq
is a $\PP^{\chi(\alpha(n))-1}$-bundle over $\cN^\perp_\alpha$. Here we have relabelled $\cE(n)$ as $\cE_n$ so that we can reuse the $\cO(n)$ notation to describe powers of the tautological bundle\footnote{This is a twisted bundle if and only if $\cE$ is twisted. The twistings cancel in $p^*\EE_n(1)$.} $\cO_{\PP(\pi_*\;\EE_n)}(-1)\into p^*\pi_*\;\EE_n$ carried by the projective bundle. By adjunction this defines the universal section
\beq{twisting}
\cO_{X\times\PP(\pi_*\;\EE_n)}\To p^*\EE_n\otimes\cO_{\PP(\pi_*\;\EE_n)}(1).
\eeq
Notice that $p^*\EE_n(1)$ is an \emph{untwisted} sheaf.
We let $I\udot$ denote the 2-term complex made out of this universal stable pair, with $\cO$ in degree 0. It is naturally $\C^*$-equivariant for the $\C^*$ action scaling the fibres of $X=K_S\to S$. We have the commutative diagram
$$
\xymatrix@=18pt{
R\hom(\cO,I\udot) \ar[d]\ar[r]& R\hom(I\udot,I\udot) \ar[dl]^{\tr}\ar[r]& R\hom(p^*\EE_n(1),I\udot)[1] \\
R\hom(\cO,\cO),}
$$
where the horizontal row is an exact triangle. Taking cones of the two downward arrows gives
$$
\xymatrix@=15pt{
R\hom(\cO,p^*\EE_n(1))[-1] \ar[r]& R\hom(I\udot,I\udot)\_0 \ar[r]& R\hom(p^*\EE_n,I\udot)[1].}
$$
Letting $\pi$ denote both projections $X\times\cP^\perp_\alpha\to\cP^\perp_\alpha$ and $X\times\cN^\perp_\alpha\to \cN^\perp_\alpha$ down $X$, as usual, we now apply $R\pi_*$. Since $n\gg0$ the first term simplifies, while the exact triangle $p^*\EE_n(1)[-1]\to I\udot\to\cO$ means the last term fits into the vertical exact triangle of the following diagram.
\beq{sick}
\xymatrix@=15pt{
&& p^*R\hom_\pi(\EE_n,\EE_n) \ar[d] \\
\pi_*p^*\EE_n(1)[-1] \ar[r]& R\hom_\pi(I\udot,I\udot)\_0 \ar[r]& R\hom_\pi(p^*\EE_n,I\udot)[1] \ar[d] \\
&& \big(\pi_*p^*\EE_n(1)\big)^\vee\otimes\t\;[-2].}
\eeq
We have used $\C^*$ equivariant Serre duality on the last term.
By stability, $p^*R^0\pi_*\;\hom(\EE_n,\EE_n)\cong\cO$ in the top right hand corner is generated by the identity, which in the above diagram maps down and right to $p^*R^0\pi_*\;\EE_n(1)$. By the relative Euler sequence, the quotient is $T_{\cP^\perp_\alpha/\cN^\perp_\alpha}$.
Its Serre dual is (the twist by $\t$ of) the connecting homomorphism $R^3\pi_*p^*R\hom(\EE_n(1),\cO)\to R^3\pi_*p^*R\hom(\EE_n,\EE_n)$ of the vertical exact triangle.
So removing these two copies of $\cO$ and $\cO\otimes\t$ gives the diagram of exact triangles
$$
\xymatrix@=15pt{
&& p^*\tau^{[1,2]}R\hom_\pi(\EE_n,\EE_n) \ar[d] \\
T_{\cP^\perp_\alpha/\cN^\perp_\alpha}[-1] \ar[r]& R\hom_\pi(I\udot,I\udot)\_0 \ar[r]& Q_1 \ar[d] \\
&& T^*_{\cP^\perp_\alpha/\cN^\perp_\alpha}\otimes\t\,[-2],}
$$
for some $Q_1$. In \cite{TT1} we construct a (split) map from $H^{\ge1}(\cO_S)\otimes\cO\ \oplus\ H^{\le1}(K_S)\;\t[-1]\otimes\cO$ to the top right hand term. There is a similar map \eqref{splits} to $R\hom_\pi(I\udot,I\udot)\_0$. These two commute with the diagram; taking cones gives
\beq{kcsp}
\xymatrix@C=2pt@R=12pt{
&&&& R\hom_\pi(p^*\EE,p^*\EE)\_\perp \ar[d] \\
T_{\cP^\perp_\alpha/\cN^\perp_\alpha}[-1] \ar[rrr]&&& R\hom_\pi(I\udot,I\udot)\_\perp \ar[r]& Q_2 \ar[d] \\
&&&& T^*_{\cP^\perp_\alpha/\cN^\perp_\alpha}\otimes\t[-2].}
\eeq
In particular, taking weight 0 parts over $(\cP^\perp_\alpha)^{\C^*}$ we
see the K-theory class of its virtual tangent bundle $R\hom_\pi(I\udot,I\udot)^{\mathrm{fix}}_\perp[1]$ is
\beq{ties}
p^*R\hom_\pi(\EE,\EE)^{\mathrm{fix}}_\perp[1]\ +\ 
T_{(\cP^\perp_\alpha)^{\C^*}/(\cN^\perp_\alpha)^{\C^*}}\ -\ 
\big(T^*_{\cP^\perp_\alpha/\cN^\perp_\alpha}\otimes\t\big)^{\mathrm{fix}},
\eeq
and the virtual normal bundle is
\beq{ties2}
N^{\vir}_{(\cP^\perp_\alpha)^{\C^*}}\ =\ p^*N^{\vir}_{(\cN^\perp_\alpha)^{\C^*}}\ +\ N_{(\cP^\perp_\alpha)^{\C^*}\!\big/\cP^\perp_\alpha|_{(\cN^\perp_\alpha)^{\C^*}}} -\ \big(T^*_{\cP^\perp_{\alpha}/\cN^\perp_\alpha}\otimes\t\big)^{\mathrm{mov}}.
\eeq
Now \eqref{ties} expresses the perfect obstruction theory of $(\cP^\perp_\alpha)^{\C^*}$ \cite{GP} in terms of the one pulled back from $\cN^\perp_\alpha$, the smooth bundle structure $p\colon\cP^\perp_\alpha\to\cN^\perp_\alpha$, and the extra final term which (by this smooth bundle structure) is a vector bundle. It follows that
$$
\big[(\cP^\perp_\alpha)^{\C^*}\big]^{\vir}\ =\ p^*\big[(\cN^\perp_\alpha)^{\C^*}\big]^{\vir}\cap e\big(\big(T^*_{\cP^\perp_{\alpha}/\cN^\perp_\alpha}\otimes\t\big)^{\mathrm{fix}}\big).
$$
Therefore
$$
P^\perp_\alpha(n)\ =\ \int_{p^*\big[(\cN^\perp_\alpha)^{\C^*}\big]^{\vir}}\frac{e\big(\big(T^*_{\cP^\perp_{\alpha}/\cN^\perp_\alpha}\otimes\t\big)^{\mathrm{fix}}\big)}{e\big(N^{\vir}_{(\cP^\perp_\alpha)^{\C^*}}\big)}\,,
$$
which by \eqref{ties2} equals
\beq{work}
\int_{p^*\big[(\cN^\perp_\alpha)^{\C^*}\big]^{\vir}}\frac{e\big(T^*_{\cP^\perp_{\alpha}/\cN^\perp_\alpha}\otimes\t\big)}{e\big(N_{(\cP^\perp_\alpha)^{\C^*}\!/\cP^\perp_\alpha|_{(\cN^\perp_\alpha)^{\C^*}}}\big)}\,p^*\!\left(\frac{1}{e\big(N^{\vir}_{(\cN^\perp_\alpha)^{\C^*}}\big)}\right).
\eeq
We integrate by first pushing down the smooth map $p\colon\big(\cP^\perp_\alpha\big)^{\C^*}\to\big(\cN^\perp_\alpha\big)^{\C^*}$. On each fibre we get
\beq{onefib}
\int_{\PP^{\C^*}}\frac{e\big(T^*_{\PP}\otimes\t\big)}{e(N_{\PP^{\C^*}/\PP})}\,,
\eeq
where $\PP=\PP^{\chi(\alpha(n))-1}$ is acted on by $\C^*$ with fixed locus $\PP^{\C^*}$. We recognise \eqref{onefib} as the computation of $e\big(T^*_{\PP}\otimes\t\big)\cap[\PP]$ by localisation to $\PP^T$; it therefore yields $(-1)^{\chi(\alpha(n))-1}\chi(\alpha(n))$ and \eqref{work} becomes
\begin{align*}
P^\perp_\alpha(n)\ &=\ (-1)^{\chi(\alpha(n))-1}\chi(\alpha(n))
\int_{[(\cN^\perp_\alpha)^{\C^*}]^{\vir}}\frac{1}{e(N^{\vir}_{(\cN^\perp_\alpha)^{\C^*}})} \\
&=\ (-1)^{\chi(\alpha(n))-1}\chi(\alpha(n))\VW_{\alpha}(S).\qedhere
\end{align*}
\end{proof}


\subsection{\for{toc}{$K_S<0$}\except{toc}{$\mathbf{K_S<0}$}}
We can also prove the conjecture when $\deg K_S<0$. While there may be strictly semistables, we find the space of Joyce-Song pairs is still smooth.

\begin{thm}\label{Kneg}
Suppose $\deg K_S<0$. Then Conjecture \ref{conj} is true and $\VW_\alpha=\vw_\alpha$.
\end{thm}

\begin{proof}
By \cite[Proposition 7.6]{TT1} the semistable Higgs pairs are all of the form $(E,0)$ for some semistable sheaf $E$ on $S$ with
\beq{smoo}
\Hom_S(E,E\otimes K_S)\ =\ 0\ =\ \Ext^2_S(E,E),
\eeq
by semistability and Serre duality respectively. We work with the corresponding sheaf $\cE=\iota_*E$ on $X$. At the level of universal sheaves $\E$ (on $S\times\cN^\perp_\alpha$) and $\EE$ (on $X\times\cN^\perp_\alpha$) we have $\EE=\iota_*\E$, which is entirely $\C^*$-fixed. Since $L\iota^*\iota_*\E\cong\E\,\oplus\,\E\otimes K_S^{-1}\t^{-1}[1]$ we have
\beq{stf}
R\hom(\EE,\EE)\ \cong\ \iota_*R\hom(\E,\E)\oplus\iota_*R\hom(\E,\E\otimes K_S)\;\t\;[-1].
\eeq
Taking weight 0 parts of \eqref{sick} therefore gives the exact triangle
$$
\xymatrix@=15pt{
\pi_*p^*\EE_n(1)[-1] \ar[r]& R\hom_\pi(I\udot,I\udot)^{\mathrm{fix}}_0 \ar[r]& p^*R\hom_\pi(\E,\E)}.
$$
Removing $H^{\ge1}(\cO_S)\otimes\cO$ and taking the long exact sequence of cohomology sheaves gives both
$$
\xymatrix@=15pt{
0\ar[r]& \frac{\pi_*p^*\E_n(1)}{p^*\hom_\pi(\E,\E)} \ar[r]& \ext^1_\pi(I\udot,I\udot)^{\mathrm{fix}}_\perp \ar[r]& p^*\ext^1_\pi(\E,\E)\_0 \ar[r]& 0}
$$
and $\ext^2_\pi(I\udot,I\udot)^{\mathrm{fix}}_\perp\rt\sim p^*\ext^2_\pi(\E,\E)_0$. The latter vanishes by basechange and \eqref{smoo}. That is, \emph{there are no fixed obstructions}, so $\cP^\perp$ is \emph{smooth and its own virtual cycle}.

Since for $\EE=\iota_*\E$ there are only terms of weights 0 and 1 in \eqref{sick}, equivariant Serre duality shows that $N^{\vir}=T^*_{\cP^\perp}\otimes\t[-1]$. Therefore
$$
P^\perp_\alpha(n)\ =\ \int_{{\cP^\perp}}e\big(T^*_{\cP^\perp}\otimes\t\big)\ =\ (-1)^{\dim\cP^\perp}e(\cP^\perp).
$$
But this is precisely the answer $P_\alpha(n)$ \eqref{plainP} that the Behrend theory gives, since by smoothness $\chi^B_{\cP^\perp}\equiv(-1)^{\dim\cP^\perp}$. Since $h^{0,2}(S)=0$ by $\deg K_S<0$ and Serre duality, $\VW_\alpha$ and $\vw_\alpha$ are determined from these two sets of pair invariants $P^\perp_\alpha(n),\,P_\alpha(n)$ by the same formulae. (That is, if $h^{0,1}(S)=0$ then the formulae \eqref{long} and \eqref{sum} are the same; if $h^{0,1}(S)>0$ then the formulae \eqref{shorter} and \eqref{be} are the same.) We deduce that $P_\alpha^\perp(n)$ satisfies the equations of Conjecture \ref{conj} with $\VW_\alpha=\vw_\alpha$.
\end{proof}

\begin{rmk}
An alternative proof is to note that since semistable sheaves are supported on $S$, the moduli space $\cP^\perp_\alpha$ is compact. Since its perfect obstruction theory is symmetric, the invariants defined by virtual cycle or weighted Euler characteristic coincide \cite{Be}.
\end{rmk}

\subsection{\for{toc}{$K_S=0$}\except{toc}{$\mathbf{K_S=0}$}} \label{K3}

In an earlier version of this paper we sketched our reasoning for believing that Conjecture \ref{conj} should be true when $S$ is a K3 surface, with $\VW_\alpha=\vw_\alpha$ in this case. This has now been proved in \cite{MT}. Namely, \cite[Proposition 7.6]{MT} gives the following identity relating the virtual invariants $P^\perp$ \eqref{bowl} to the weighted Euler characteristic invariants $\widetilde P$ of \eqref{welo}.
\beq{it}
\sum P^\perp_\alpha(n)q^\alpha\ =\ -\log\left(1+\sum\widetilde P_\alpha(n)q^\alpha\right),
\eeq
where both sums are over all $\alpha\ne0$ which are multiples of a fixed primitive class $\alpha_0.$
Since the log on the right hand side of \eqref{it} turns  the formula \eqref{munch} into \eqref{shorter}, we indeed find that Conjecture \ref{conj} holds with $\VW_\alpha=\vw_\alpha$. \medskip

The idea of the proof of \eqref{it} is to start with $Y=S\times E$, where $E$ is an elliptic curve. To get nonzero invariants counting pairs we divide the moduli space by the translation action of $E$ and use Oberdieck's symmetric reduced obstruction theory \cite{Ob}. Since the moduli space is compact, the invariants defined by virtual cycle or weighted Euler characteristic coincide. In appropriate notation,
$$
P^\perp_\alpha(Y/E,n)\ =\ \widetilde P_\alpha(Y/E,n).
$$
On the left hand side we use Jun Li's degeneration formula for virtual cycles as $E$ degenerates to a rational nodal curve. This has a $\C^*$ action; applying virtual localisation ultimately gives the left hand side of \eqref{it}. On the right hand side we work with weighted Euler characteristics, using a simple gluing argument to compare the moduli spaces of Joyce-Song pairs supported set-theoretically on one K3 fibre of either $X$ or $Y$. An elementary calculation of Euler characteristics of configuration spaces of points on a punctured elliptic curve then gives the right hand side of \eqref{it}.
\medskip

Here we make do with a suggestive explicit calculation of the virtual theory $\VW_\alpha$ in some examples, to illustrate \eqref{it} and contrast with the calculations of the Kai theory $\vw_\alpha$ on K3 in Section \ref{Kthree}.

We work in $SU(r)$ Vafa-Witten theory with $\C^*$-fixed Higgs pairs which have the least possible degeneracy in their Higgs field (i.e. it is a Jordan block). Equivalently we work with $\C^*$-fixed torsion sheaves on $X=K_S$ with the largest possible scheme-theoretic support: the $r$ times thickening $rS$ of the zero section $S\subset X$.

Semistable sheaves with support $rS$ are all of the form
$$
\cE\ =\ (\pi^*I_Z)\otimes\cO_{rS}
$$
for some ideal sheaf $I_Z\in S^{[k]}$. Thus $E\cong I_Z\oplus I_Z\t^{-1}\oplus\cdots\oplus I_Z\t^{-(r-1)}$ with $\det E=\cO_S$ and class $\alpha=(r,rk)$, and
$$
H^0_X(\cE(n))\ \cong\ H^0_S(I_Z(n))\otimes\!\big(\C\oplus\t^{-1}\oplus\cdots\oplus\t^{-(r-1)}\big).
$$
Any $\C^*$-fixed section $s$ defining a Joyce-Song stable pair must lie in the first summand by stability, so $(\cP^\perp)^{\C^*}$ is a $\PP(H^0(I_Z(n)))$-bundle over $S^{[k]}$. Then by a similar calculation to \eqref{sick}, at a single point $I\udot=\{\cO_X(-n)\rt{s}\cE\}$ for simplicity, we find that in $\C^*$-equivariant K-theory 
\begin{multline*}
-R\Hom_X(I\udot,I\udot)\_\perp\ =\ \\
\frac{H^0(\cE(n))}{\aut(\cE)}-\left(\frac{H^0(\cE(n))}{\aut(\cE)}\right)^{\!\!*}\!\otimes\t
+\Ext^1_X(\cE,\cE)\_\perp-\Ext^2_X(\cE,\cE)\_\perp\,.
\end{multline*}
Replacing $\cE$ by its natural resolution $\pi^*I_Z(-rS)\to\pi^*I_Z$ gives
\begin{multline*}
R\Hom_X(\cE,\cE)\ =\ R\Hom_S(I_Z,I_Z)\otimes\!\big(\C\oplus\t^{-1}\oplus\cdots\oplus\t^{-(r-1)}\big)
\\ \oplus\ R\Hom_S(I_Z,I_Z)\otimes\!\big(\t^r\oplus\cdots\oplus\t\big)[-1],
\end{multline*}
so that
$$
\Ext^1_X(\cE,\cE)\_\perp\ \cong\ T_ZS^{[k]}\otimes\!\big(\C\oplus\t^{-1}\oplus\cdots\oplus\t^{-(r-1)}\big)
\ \oplus\ \big(\t^r\oplus\cdots\oplus\t^2\big),
$$
while $\Ext^2_X(\cE,\cE)\_\perp\cong\Ext^1_X(\cE,\cE)_\perp^*\otimes\t$. Now $\Aut(\cE)\cong\C^*\ltimes(\t^{-1}\oplus\cdots\oplus\t^{-(r-1)})$ and
$$
\frac{H^0(\cE(n))}{\aut(\cE)}\ \cong\ \frac{H^0(I_Z(n))}{\langle s\rangle}\otimes\!\big(\C\oplus\t^{-1}\oplus\cdots\oplus\t^{-(r-1)}\big)
$$
is $T_{\PP(H^0(I_Z(n)))}\otimes(\C\oplus\t^{-1}\oplus\cdots\oplus\t^{-(r-1)})$ when we work in a family using the $\cO(1)$ twisting of the universal section as in \eqref{twisting}. Putting everything together, the pairs invariant $P^\perp_\alpha(n)$ is
$$
\int_{(\cP^\perp)^{\C^*}}\frac{e(T^*\t)e(T^*\t^2)\cdots e(T^*\t^r)e(\t^{-1}\oplus\cdots\oplus\t^{-(r-1)})}
{e(T\t^{-1})e(T\t^{-2})\cdots e(T\t^{-(r-1)})e(\t^r\oplus\cdots\oplus\t^2)}\,,
$$
where $T=T_{(\cP^\perp)^{\C^*}}=T_{S^{[k]}}+T_{(\cP^\perp)^{\C^*}/S^{[k]}}$. Since $e(A^*)/e(A)=(-1)^{\rk A}$ some cancellation gives
\begin{multline*}
\int_{(\cP^\perp)^{\C^*}}(-1)^{p(r-1)}e(T^*\t^r)\frac{(-t)(-2t)\cdots(-(r-1)t)}{(2t)\cdots(rt)}
\ =\ \\
(-1)^{pr}e\big(\PP(H^0(I_Z(n)))\big)e\big(S^{[k]}\big)\frac{(-1)^{r-1}}r\,,
\end{multline*}
where $p=\dim(\cP^\perp)^{\C^*}=\chi(I_Z(n))-1+2k$.
Since $\chi(\alpha(n))=r\chi(I_Z(n))$ the final result is
\beq{fr}
P_\alpha^\perp(n)\ =\ (-1)^{\chi(\alpha(n))-1\,}\frac{\chi(\alpha(n))}{r^2}e\big(S^{[k]}\big).
\eeq
This satisfies Conjecture \ref{conj}. That is, it is an instance of our universal formula \eqref{shorter} which is the \emph{logarithm} of the Joyce-Song universal formula \eqref{munch} satisfied by the invariants $\vw_\alpha$. It contributes $e(S^{[k]})/r^2$ to the Vafa-Witten invariant $\VW_{r,rk}(S)$ and
$$
\sum_{k=0}^\infty\frac{e(S^{[k]})}{r^2}q^{rk}\ =\ \frac1{r^2}q^r\eta(q^r)^{-24}
$$
to its generating series.

This matches the second term of \cite[Equation 5.38]{VW} (with $g-1:=c_1^2(S)=0$), and the first term ($d=1$) of the generating series \eqref{generalr} of $\vw_\alpha$ invariants. However, the $\vw_\alpha$ calculation was somewhat different, with contributions from different components giving different answers for the pairs invariants, but then the different universal formulae (\ref{munch}, \ref{shorter}) these satisfy lead to the same Vafa-Witten invariants $\vw_\alpha=\VW_\alpha$.

This can be seen even in the simplest example with $k=0$ and $r=2$, so that $\alpha=(2,0)$. Then the moduli space of $\C^*$-fixed Joyce-Song pairs has two components --- one a space of sections of $\cO_{2S}(n)$, and one a Grassmannian $\Gr(2,H^0(\cO_S(n)))$ of pairs with underlying sheaf $\cO_S^{\oplus2}$. The virtual localisation calculation \eqref{fr} gives
\beq{nos}
P^\perp_\alpha(n)\ =\ \frac14(-1)^{\chi(\alpha(n))-1}\chi(\alpha(n))\ =\ -\frac14\chi(\alpha(n))
\eeq
for the first component and \emph{zero} for the second.\footnote{An easy calculation shows its $\C^*$-fixed obstruction bundle is the bundle of trace-free endomorphisms of the universal subbundle on $\Gr(2,H^0(\cO_S(n)))$. But this has $c_3=0$ so the localised virtual cycle vanishes.} By contrast, the topological Euler characteristic of the first and second components are
$$
\chi(\alpha(n))-\chi\!\left(\frac\alpha2(n)\right) \quad\mathrm{and}\quad \frac12\chi\!\left(\frac\alpha2(n)\right)\left(\chi\!\left(\frac\alpha2(n)\right)-1\right)
$$
respectively. ($\chi^B\equiv1$ here, by Remark \ref{kaistack}, so we can ignore the weighting.) These are very different from \eqref{nos}, as is their sum $P_\alpha(n)$. But putting $P^\perp_\alpha(n)$ into \eqref{munch} and $P_\alpha(n)$ into \eqref{shorter} respectively gives the same Vafa-Witten invariants $\VW_\alpha=\vw_\alpha$. \medskip

Combining the result \eqref{it} of \cite{MT} with our Toda-based calculations of Section \ref{Kvw3} we get the virtual localisation version of Theorem \ref{Kvw2}.

\begin{thm}
Conjecture \ref{conj} holds for $S$ a K3 surface. 
The resulting rank $r$ trivial determinant Vafa-Witten invariants have generating series
$$
\sum_{c_2}\VW_{r,c_2}q^{c_2}\ =\ \sum_{d|r}\frac{d}{r^2}q^r
\sum_{j=0}^{d-1}\eta\Big(e^{\frac{2\pi ij}d}q^{\frac r{d^2}}\Big)^{-24}.
$$
\end{thm}

\subsection{\for{toc}{$K_S>0$}\except{toc}{$\mathbf{K_S>0}$}}
Here we describe a rather trivial calculation on general type surfaces $S$. Nonetheless it is again suggestive that the result takes the correct form to satisfy our Conjecture. 

Take a surface $S$ with $h^{0,1}(S)=0$ and $h^{0,2}(S)>0$ and charge $\alpha=(2,0,0)$. There is a $\PP(\Gamma(K_S))\ni[\sigma]$ of strictly semistable $\C^*$-fixed trivial determinant trace-free Higgs pairs of the form
$$
E\ =\ \cO_S\oplus\cO_S\t^{-1}, \qquad \phi=\mat00\sigma0.
$$
We use the same deformation theory of Higgs pairs as in \cite[Section 8.1]{TT1}.
In the obvious notation, the map $[\ \cdot\ ,\phi]\colon\Hom(E,E)\to\Hom(E,E\otimes K_S)$ is
\begin{align*}
\mat{\C}{\t}{\t^{-1}}{\C}\,&\To\ \mat{\Gamma(K_S)\t}{\Gamma(K_S)\t^2}{\Gamma(K_S)}{\Gamma(K_S)\t}, \\
\mat abcd\ &\Mapsto\ \mat{b\sigma}0{(d-a)\sigma}{-b\sigma}.
\end{align*}
Passing to the corresponding trace-free groups gives a map with kernel $\End_0(E,\phi)=\t^{-1}$ and cokernel
$$
\Ext^1(\cE_\phi,\cE_\phi)\_\perp\ =\ \Gamma(K_S|_C)\t\oplus\Gamma(K_S)\t^2\oplus\Gamma(K_S|_C),
$$
where $C\subset S$ is the divisor of $\sigma$. Done properly, over the family $\PP(\Gamma(K_S))$ with the twisting $s\in\Gamma\big(K_S\boxtimes\cO_{\PP(\Gamma(K_S))}(1)\big)$ as in \eqref{twisting}, this is
\beq{extp}
T_{\PP(\Gamma(K_S))}\t\ \oplus\ \Gamma(K_S)\!\otimes\!\cO_{\PP(\Gamma(K_S))}(1)\t^2\ \oplus\ T_{\PP(\Gamma(K_S))}.
\eeq
Now consider the Joyce-Song pairs $I\udot=\{\cO_X(-n)\rt{s}\cE_\phi\}$. Their moduli space $\cN^\perp$ restricts over the $\C^*$-fixed moduli space $\PP(\Gamma(K_S))$ to a bundle with fibre the quotient of
\beq{dubul}
\PP\big(H^0(\cO_S(n))\oplus H^0(\cO_S(n))\t^{-1}\big)\take\PP\big(H^0(\cO_S(n))\t^{-1}\big)
\eeq
by the obvious action of $\t^{-1}$. Just as in \eqref{tanh} we find that
$$
\Ext^1(I\udot,I\udot)\_\perp\ =\ T_{\cN^\perp\!/\;\PP(\Gamma(K_S))}
\,\oplus\,\Ext^1(\cE_\phi,\cE_\phi)\_\perp.
$$
Combining this with \eqref{extp} and \eqref{dubul}, we find that at a $\C^*$-fixed Joyce-Song pair $s=(s_1,0)$ the tangent space $\Ext^1(I\udot,I\udot)\_\perp$ is
\begin{align} \nonumber
T_{\PP(H^0(\cO_S(n)))}\ &\oplus\ T_{\PP(H^0(\cO_S(n)))}\t^{-1} \\ \label{mess} &\oplus\ 
T_{\PP(\Gamma(K_S))}\t\ \oplus\ \Gamma(K_S)\!\otimes\!\cO_{\PP(\Gamma(K_S))}(1)\t^2\ \oplus\ T_{\PP(\Gamma(K_S))}.
\end{align}
The first and last terms give the fixed tangent space expected. Recalling the Serre duality $\Ext^2(I\udot,I\udot)\_\perp
\cong\Ext^1(I\udot,I\udot)^*_\perp\t$, we find the third term gives a fixed obstruction bundle $T^*_{\PP(\Gamma(K_S))}$. The virtual cycle is therefore its Euler class
\beq{lst}
\big[(\cP^\perp)^{\C^*}\big]^{\vir}\ =\ (-1)^{h^0(K_S)-1}h^0(K_S)\big[\PP(H^0(\cO_S(n)))\big],
\eeq
where $\PP(H^0(\cO_S(n)))$ is the fibre over a point of $\PP(\Gamma(K_S))$. On these fibres the moving part of \eqref{mess} simplifies to
$$
T_{\PP(H^0(\cO_S(n)))}\t^{-1}\ \oplus\ 
\t^{\oplus\,h^0(K_S)-1}\ \oplus\ (\t^2)^{\oplus\,h^0(K_S)}.
$$
Similarly the moving part of $\Ext^2(I\udot,I\udot)\_\perp
\cong\Ext^1(I\udot,I\udot)^*_\perp\t$ is
$$
T^*_{\PP(H^0(\cO_S(n)))}\t\ \oplus\ 
T^*_{\PP(H^0(\cO_S(n)))}\t^2\ \oplus\ 
(\t^{-1})^{\oplus\,h^0(K_S)}\ \oplus\ \t^{\oplus\,h^0(K_S)-1}.
$$
Together these give the virtual normal bundle, with equivariant Euler class
\beqa
e(N^{\vir}) &=& \frac{e\big(T_{\PP(H^0(\cO_S(n)))}\t^{-1}\big)\cdot t^{h^0(K_S)-1}\cdot(2t)^{h^0(K_S)}}{e\big(T^*_{\PP(H^0(\cO_S(n)))}\t\big)e\big(T^*_{\PP(H^0(\cO_S(n)))}\t^2\big)\cdot(-t)^{h^0(K_S)}\cdot t^{h^0(K_S)-1}} \\
&=& (-1)^{\chi(\cO_S(n))-1+h^0(K_S)}2^{h^0(K_S)}\frac1{e\big(T^*_{\PP(H^0(\cO_S(n)))}\t^2\big)}\,.
\eeqa
Integrating $1/e(N^{\vir})$ over the virtual cycle \eqref{lst} therefore gives
$$
P^\perp_\alpha(n)\ =\ -2^{-h^0(K_S)}h^0(K_S)\chi(\cO_S(n)).
$$
Since $\chi(\alpha(n))=2\chi(\cO_S(n))$ we see that this fits our conjecture \eqref{shorter} perfectly with contribution
\beq{end}
\VW_\alpha\ =\ \frac{h^0(K_S)}{2^{\;h^0(K_S)+1}}\,.
\eeq
The other $\C^*$-fixed component contains only (Joyce-Song pairs for) the semistable bundle $\cO_S^{\oplus2}$ with $\phi=0$, and contributes nothing due to the trivial $H^0(K_S)$ piece of the obstruction space $\Ext^2_\perp$. Therefore $P^\perp_\alpha(n)$ has \emph{no terms quadratic in $\chi(\alpha(n))$}, as predicted by the conjecture and in contrast to the invariants $P_\alpha(n)$ defined by weighted Euler characteristic. Finally we note that since $h^{0,1}(S)=0$, the denominator of \eqref{end} is $2^{\chi(\cO_S)}$, which appears also in the denominator of the first line of \cite[Equation 5.37]{VW}.

\bibliographystyle{halphanum}
\bibliography{References}

\end{document}